\documentclass{article}

\usepackage{xcolor}

\usepackage{verbatim}

\usepackage{amsfonts}

\usepackage{amsthm}

\usepackage{amssymb}

\usepackage{amsmath}

\usepackage{mathabx}

\usepackage{enumerate}

\usepackage[all]{xy}

\usepackage{graphicx}

\usepackage{centernot}

\usepackage{stmaryrd}

\usepackage{mathrsfs}

\usepackage[pagebackref,colorlinks,linkcolor=blue,citecolor=blue,urlcolor=blue,hypertexnames=true]{hyperref}

\usepackage[pagebackref,hypertexnames=true]{hyperref}

\newcommand{\N}{\mathbb{N}}

\newcommand{\Z}{\mathbb{Z}}

\newcommand{\R}{\mathbb{R}}

\newcommand{\Q}{\mathbb{Q}}

\newcommand{\C}{\mathbb{C}}

\newcommand{\Hawaii}{Hawai\kern.05em`\kern.05em\relax i}

\newcommand{\tor}{\text{t}}

\theoremstyle{plain}
\newtheorem{theorem}{Theorem}[section]
\newtheorem{lemma}[theorem]{Lemma}
\newtheorem{corollary}[theorem]{Corollary}
\newtheorem{proposition}[theorem]{Proposition}

\newtheorem{definition-theorem}[theorem]{Definition / Theorem}

\newtheorem*{conjecture*}{Conjecture}
\newtheorem*{theorem*}{Theorem}

\theoremstyle{definition}
\newtheorem{definition}[theorem]{Definition}
\newtheorem{example}[theorem]{Example}
\newtheorem{examples}[theorem]{Examples}

\theoremstyle{remark}
\newtheorem{remark}[theorem]{Remark}

\newtheorem*{example*}{Example}  
\newtheorem*{remark*}{Remark}

\begin{document}
\title{A secondary pairing between $K$-theory and $K$-homology, relative eta invariants, and zeta maps}
\author{Rufus Willett}

\maketitle

\begin{abstract}
The $K$-homology groups of a $C^*$-algebra are receptacles for information from topology, operator algebra theory, and representation theory.  For applications, one often wants to know if two $K$-homology classes are the same: the simplest way to deduce this is typically via the `primary' pairing between $K$-homology and the dual theory ($K$-theory).  However, this pairing will typically miss some information: for example, it cannot detect torsion elements of $K$-homology.

In this paper, we introduce a `secondary' pairing between subgroups of $K$-homology and $K$-theory that takes values in $\Q/\Z$.  In good cases we show that this pairing will detect all the classes in $K$-homology that are missed by the primary pairing.  We then relate our secondary pairing to the relative eta invariants of Atiyah-Patodi-Singer, and to the Thomsen exact sequence and zeta maps from $C^*$-algebra classification theory.
\end{abstract}

\tableofcontents

\section{Introduction}



For a (separable) $C^*$-algebra $A$ there are two $K$-homology groups $K^0(A)$ and $K^1(A)$ up to isomorphism; these are topological abelian groups that form an important `homological' invariant of the $C^*$-algebra.  The topology on $K^i(A)$ is not always Hausdorff\footnote{It is Hausdorff if $K^i(A)$ is finitely generated, which will happen in many interesting examples such as if $A=C(X)$ for $X$ a finite CW complex.} and for many applications it is really the maximal Hausdorff quotient that is important.  Following R\o{}rdam \cite{Rordam:1995aa} and Dadarlat \cite{Dadarlat:2005aa}, we will denote the maximal Hausdorff quotient group of $K^i(A)$ by $KL^i(A)$, i.e.\ $KL^i(A)$ is the quotient of $K^i(A)$ by the closure of the trivial subgroup.  

Cycles for $K(L)$-homology arise in a variety of applications in topology, geometry, operator algebra theory, and representation theory.  One is often interested in knowing when two given cycles define the same class, or equivalently, when a given cycle defines the zero class.  Now, the $K$-homology group $KL^i(A)$ pairs with the $K$-theory group $K_i(A)$, and the latter is much easier to compute.  A powerful general method to tell if a cycle $x$ for the $K$-homology group $KL^i(A)$ defines the zero class is to look at the map  
\begin{equation}\label{x*}
x_*:K_i(A)\to \Z
\end{equation}
induced by its pairing with $K$-theory.  Computing the map $x_*$ is often possible in interesting cases, for example using index theory.  However, we cannot hope to get information this way if $x\in KL^i(A)$ is in the kernel $kl^i(A)$ of the canonical pairing map
\begin{equation}\label{intro gamma}
\gamma:KL^i(A)\to \text{Hom}(K_i(A),\Z),\quad x\mapsto x_*.
\end{equation}


The main goal of this paper is to define a secondary pairing between $kl^i(A)$ and the torsion subgroup $\tor K_{1-i}(A)$ of the $K$-theory group $K_{1-i}(A)$ in the opposite degree, that takes values in $\Q/\Z$.  Here is our basic result.

\begin{theorem}\label{sec pair intro}
Let $A$ be a separable $C^*$-algebra.  Then for $i\in \{0,1\}$ there is a pairing map 
$$
\delta:kl^i(A)\to \text{Hom}(\tor K_{1-i}(A),\Q/\Z),
$$
natural in $A$.  

If $A$ satisfies the UCT of Rosenberg and Schochet \cite{Rosenberg:1987bh}, then $\delta$ is a homeomorphic isomorphism when $\text{Hom}(\tor K_{1-i}(A),\Q/\Z)$ is equipped with the topology of pointwise convergence.  In particular, a cycle $x$ for $KL^i(A)$ defines the zero class if and only if $\gamma(x)$ (see line \eqref{intro gamma}) and $\delta(x)$ are both zero.
\end{theorem}

There are analogous very classical constructions in homology theory: for example, the linking numbers relating torsion classes in manifold topology in the sense discussed in \cite[page 523]{Kervaire:1963aa}\footnote{The definition of linking number implicitly invokes Poincar\'{e} duality to define a pairing of the torsion subgroup of homology with itself, rather than a pairing between homology and cohomology.} and the discussion of torsion invariants in \cite[Section 2]{Morgan:1974aa}.  


We treat this pairing in three different ways, which we show are equivalent: the most complicated (but probably also most powerful) relates the secondary pairing to the category of Bockstein operations and the universal multicoefficient theorem of Dadarlat-Loring \cite{Dadarlat:1996aa}: see Theorem \ref{totk} and Corollary \ref{uct pair} below for the most general statements.

We should note that parts of Theorem \ref{sec pair intro} are very classical, appearing implicitly in the work of Brown \cite{Brown:1984rx} and Brown-Douglas-Fillmore \cite{Brown:1973aa}.  For example, the statement that $\delta$ is a homeomorphic isomorphism in the presence of the UCT is the same as saying that $kl^i(A)$ identifies with the Pontrjagin dual of $\tor K_{1-i}(A)$, which is in turn related to the discussion on \cite[pages 62-3]{Brown:1984rx}. Moreover, the fact that $kl^i(A)$ is the maximal compact subgroup of $K^i(A)$ (which is a straightforward corollary of the statement about Pontrjagin duals) essentially appears on \cite[page 977]{Brown:1973aa}.  Our results are more general and fully developed than these, and can be thought of as providing full details, and a fuller context, to these earlier hints.

\subsection{Applications and examples}

We give some cases where the secondary pairing from Theorem \ref{sec pair intro} can be explicitly computed.  The first main example comes from geometry: the relative eta invariant of Atiyah-Patodi-Singer \cite{Atiyah:1975aa}.  Here is a sample theorem, based closely on the work of Antonini-Azzali-Skandalis \cite{Antonini:2014aa}.  We will not state the definitions here, but suffice to say that the relative eta invariant is a measure of spectral asymmetry of a self-adjoint elliptic differential operator defined using a sort of zeta function regularization; it is well-studied in differential topology and mathematical physics.

The following theorem refines and generalizes the results of \cite[Section 6]{Willett:2024aa}, which was one of our main motivations for this paper.

\begin{theorem}\label{intro eta}
Let $M$ be a closed manifold.  Let $D$ be a Dirac-type operator on $M$ defining a torsion class $[D]$ in the $K$-homology group $KL^1(C(M))$, whence is in $kl^1(C(M))$.  Let $V$ and $W$ be flat bundles over $M$ of the same dimension\footnote{These assumptions imply that the class $[V]-[W]$ in the $K$-theory group $K_0(C(M))$ is torsion.}.  Then
$$
\delta[D]([V]-[W])=\rho_{V,W}(D)
$$
where the right hand side is the relative eta invariant of Atiyah-Patodi-Singer.
\end{theorem}

This implies, for example, that if a Dirac-type operator $D$ on a closed manifold $M$ defines a torsion class $[D]$ in $K^1(C(M))$, then for any flat bundles $V$ and $W$ on $M$, the relative eta invariant $\rho_{V,W}(D)$ depends only on the isomorphism classes of $V$ and $W$ as \emph{topological} vector bundles.  This should be contrasted with the case when $[D]$ is not torsion, in which case $\rho_{V,W}(D)$ depends on the isomorphism classes of $V$ and $W$ as \emph{flat} vector bundles; the latter is much finer information.  We give examples highlighting this distinction, and other examples where the pairing is non-trivial in this context.  We also prove a version of the above theorem for finite CW complexes using the Baum-Douglas model of $K$-homology; this is motivated by work in progress around the Baum-Connes conjecture and flat bundles.

The other main example concerns finite-dimensional representations.  Again, we do not state it in maximal generality, but just give a flavor that is still strong enough to give some non-trivial applications.

\begin{theorem}\label{intro zeta}
Let $A$ be a separable unital $C^*$-algebra, and let $\pi,\sigma:A\to M_n(\C)$ be representations that induce the same map $K_0(A)\to \Z$, so there is a well-defined class $[\pi]-[\sigma]$ in $kl^0(A)$.  Assume moreover that $u\in A$ is a unitary with finite order in the abelianization $U(A)_{ab}$ of the unitary group of $A$\footnote{This implies that $[u]$ is a torsion class in $K_1(A)$.}.  Then 
$$
\delta([\pi]-[\sigma])[u]=\frac{1}{2\pi i} \log\det(\pi(u)\sigma(u^*)).
$$
\end{theorem}

The proof goes through a fundamental exact sequence of Thomsen \cite{Thomsen:1995aa}, and the zeta maps introduced independently by Gong-Lin-Niu \cite[Remark 6.6]{Gong:2023aa}, and Carri\'{o}n-Gabe-Schafhauser-Tikuisis-White \cite[Section 3.1]{Carrion:2020aa}.

We give examples coming from group $C^*$-algebras where the pairing above is non-trivial, and gives interesting information.  We hope to develop this and relate it to the Baum-Connes conjecture and Theorem \ref{intro eta} in subsequent work.

\subsection{Outline of the paper}

Section \ref{prelim sec} recalls notation and sets up conventions; mainly this concerns operator $K$-theory, $K$-homology, and $KK$-theory.  Section \ref{gen pair sec} is the heart of the paper, describing our secondary pairing in detail.  It is split into three subsections each giving a different approach to the secondary pairing: Subsections \ref{qz sec} and \ref{ext sec} are relatively elementary (the background needed is forty to fifty years old at this point), while Subsection \ref{mc sec} is more sophisticated, relating the secondary pairing to the Bockstein operations and universal multicoefficient theorem from \cite{Dadarlat:1996aa} (so needs background that is `only' thirty years old).

Sections \ref{eta sec} \ref{tz sec} discuss applications corresponding to Theorems \ref{intro eta} and \ref{intro zeta} respectively.  Section \ref{eta sec} starts with a background discussion leading up to the work of Antonini-Azzali-Skandalis \cite{Antonini:2014aa}.  We then use their work to derive our main theorems.  Section \ref{tz sec} starts with background discussion about the Thomsen exact sequence from \cite{Thomsen:1995aa}, and the zeta map from \cite{Gong:2023aa,Carrion:2020aa}; we then use this to derive a more general version of Theorem \ref{intro zeta} above.  Sections \ref{eta sec} and \ref{tz sec} both contain examples showing the secondary pairings can be non-trivial in the relevant contexts; both are in some sense related to the Baum-Connes conjecture, and we hope to say more about this in future work (and / or we hope others say more in future work).

\subsection{Acknowledgements}

I am grateful for support from the US NSF (DMS 2247968) and Simons foundation (MP-TSM-00002363) during the writing of the paper.  I would like to thank Jos\'{e} Carri\'{o}n, Jamie Gabe, Guihua Gong, Huaxin Lin, Mikkel Munkholm, Chris Schafhauser, and Stuart White for their patience during various illuminating conversations on the $K$-theoretic inputs to $C^*$-classification theory, and Robin Deeley, Hongzhi Liu, and Hang Wang for some discussion on the geometric side.  I would like to thank Runjie Hu for pointing out the reference \cite{Morgan:1974aa}.

\section{Preliminaries and notation}\label{prelim sec}

In this section, we introduce notation and conventions, and recall background.  

We will generally reserve the notation $A$ for an abstract $C^*$-algebra.  For $C^*$-algebras $A$ and $B$, we write $A\otimes B$ for their minimal tensor product.  For a group $\Gamma$, we write $C^*\Gamma$ for the maximal group $C^*$-algebra.  We write $SA:=C_0(0,1)\otimes A$ for the suspension of a $C^*$-algebra.  

We will use $K$-theory, $K$-homology, and $KK$-theory of $C^*$-algebras throughout.  For general background see \cite{Blackadar:1998yq}.  

We write $K_0(A)$, and $K_1(A)$ for the even and odd $K$-theory groups of a $C^*$-algebra $A$, and $K^0(X)$ and $K^1(X)$ for the even and odd $K$-theory groups of a compact Hausdorff space $X$ (so $K^i(X)=K_i(C(X))$ for $i\in \{0,1\}$).  Similarly for $i\in \{0,1\}$, we write $K^i(A)$ for the $K$-homology group of a $C^*$-algebra, and $K_i(X)$ for the $K$-homology of a space.  We write $K_*(A)$ for the ordered pair of groups $(K_0(A),K_1(A))$, and $\text{Hom}(K_*(A),K_*(B))$ for the group of pairs of homomorphisms $(\phi_0,\phi_1)$ with $\phi_i:K_i(A)\to K_i(B)$ a homomorphism; similarly for $K$-homology.

For separable $C^*$-algebras $A$ and $B$, we write $KK(A,B)$ for the Kasparov $KK$-theory group.  This generalizes $K$-theory and $K$-homology in the sense that there are canonical identifications $KK(A,\C)=K^0(A)$ and $KK(\C,A)=K_0(A)$.  We write $KK^1(A,B)$ for $KK(SA,B)$, or equivalently, $KK(A,SB)$.  Addition of indices on $K$-theory, $K$-homology, and $KK$-theory groups should always be understood modulo two, and $KK^0(A,B)$ is the same as $KK(A,B)$.

If $\alpha\in KK^i(A,B)$ and $\beta\in KK^j(B,C)$, then we will write $\beta\circ \alpha$ for their Kasparov product in $KK^{i+j}(A,C)$; this is probably more commonly written ``$\alpha\otimes \beta$'' (note the switch in order!), but it is more useful for us to have notation that mimics composition of morphisms.  Indeed, for many applications it is useful to think of the $KK$-groups as the morphism sets in an additive category with objects $C^*$-algebras and we will often do this.  

If $A$, $B$, and $C$ are $C^*$-algebras there is a canonical map 
$$
KK(A,B)\to KK(A\otimes C,B\otimes C)
$$
denoted $\alpha\mapsto \alpha\otimes \text{id}_C$.  As a special case of the Kasparov product, this induces a map 
$$
KK(A,B)\to \text{Hom}(K_0(A\otimes C),K_0(B\otimes C)).
$$
Similarly, there is a pairing map
\begin{equation}\label{gamma map}
\gamma:KK(A,B)\to \text{Hom}(K_*(A),K_*(B))
\end{equation}
where for $\alpha\in KK(A,B)$ the component of $\gamma(\alpha)$ in $\text{Hom}(K_0(A),K_0(B))$ is given by Kasparov product, and the component of $\gamma(\alpha)$ in $\text{Hom}(K_1(A),K_1(B))$ is given by the Kasparov product with $\alpha\otimes\text{id}_{C_0(0,1)}$ and the identification $K_0(SC)=K_1(C)$ for any $C^*$-algebra $C$.

Specializing to $B=\C$ and making appropriate use of suspensions gives a pairing 
\begin{equation}\label{pair map}
\gamma:K^i(A)\to \text{Hom}(K_i(A),\Z)
\end{equation}
for each $i\in \{0,1\}$; we will also use the following notation for this pairing
$$
K^i(A)\times K_i(A)\to \Z,\quad (x,y)\mapsto \langle x,y\rangle.
$$

There is a canonical (complete, but possibly non-Hausdorff) topology on each $KK(A,B)$ making it into a topological group: see Dadarlat's paper \cite{Dadarlat:2005aa} for the definitive version\footnote{If one just wants the topology on $K$-homology (which will almost always be enough for us), this goes back to Brown-Douglas-Fillmore: see \cite[Section 5]{Brown:1973aa} and \cite[Section 8.5]{Brown:1977qa}.}.  Following \cite[Section 5]{Dadarlat:2005aa}, we define $KL(A,B)$ to be the maximum Hausdorff quotient of $KK(A,B)$, i.e.\ the quotient of $KK(A,B)$ by the closure of zero.  In the special case of $K$-homology, we write $KL^i(A)$ for this maximal Hausdorff quotient; note that the topology on $KK$ specializes to the discrete topology on $KK(\C,A)=K_0(A)$, so we do not get anything new.  

We will equip the group $\text{Hom}(K_*(A),K_*(B))$ with the topology of pointwise convergence.  As the Kasparov product is continuous (see \cite[Theorem 3.5]{Dadarlat:2005aa}) and $\text{Hom}(K_*(A),K_*(B))$ is Hausdorff, the homomorphism $\gamma$ of line \eqref{gamma map} descends to a homomorphism 
$$
\gamma:KL(A,B)\to \text{Hom}(K_*(A),K_*(B))
$$
and similarly the pairing of line \eqref{pair map} induces 
\begin{equation}\label{pair l map}
\gamma:KL^i(A)\to \text{Hom}(K_i(A),\Z).
\end{equation}

If $A$ is a $C^*$-algebra in the so-called \emph{UCT class} of \cite{Rosenberg:1987bh} (equivalently $A$ is $KK$-equivalent to a commutative $C^*$-algebra), then the map $\gamma$ from line \eqref{gamma map} first into a short exact sequence
\begin{equation}\label{uct ses}
0\to \text{Ext}(K_{*+1}(A),K_*(B))\to KK(A,B)\stackrel{\gamma}{\to} \text{Hom}(K_*(A),K_*(B))\to 0,
\end{equation}
for any separable $C^*$-algebra $B$.  Here $\text{Ext}$ denotes the usual bifunctor of homological algebra, and the notation ``$\text{Ext}(K_{*+1}(A),K_*(B))$'' means that we switch degrees, so this group consists of pairs $(e,f)$ with $e\in \text{Ext}(K_0(A),K_1(B))$ and $f\in \text{Ext}(K_1(A),K_0(B))$.  For us, it will be enough to know that the UCT class includes all (separable) commutative $C^*$-algebras, and is closed under $KK$-equivalence, extensions, and (minimal)\footnote{or maximal, but we won't use this.} tensor products: this can either be found in the original paper of Rosenberg-Schochet \cite{Rosenberg:1987bh}, or is a simple application of their methods.

\section{A secondary pairing between $K$-theory and $K$-homology}\label{gen pair sec}

In this section we introduce a secondary pairing between $K$-theory and $K$-homology of a $C^*$-algebra $A$.  Indeed, recall that the `primary' pairing between $K$-theory and $K$-homology consists of homomorphisms
\begin{equation}\label{gamma pair}
\gamma:K^i(A)\to \text{Hom}(K_i(A),\Z)
\end{equation}
for $i\in \{0,1\}$ as in line \eqref{pair map}.  Our secondary pairing will be defined only when the primary pairing vanishes: if we write $k^i(A)$ for the kernel of $\gamma$ and $tK_i(A)$ for the torsion subgroup of $K_i(A)$, then our secondary pairing will be a homomorphism\footnote{We use the ``$\delta$'' as this is the next letter after $\gamma$.}
$$
\delta:k^i(A)\to \text{Hom}(\tor K_{i+1}(A),\Q/\Z).
$$
Our applications concern only the cases that $A$ is a group $C^*$-algebra or the continuous functions on a compact Hausdorff space, but we work in general in this section.

The section is split into three subsections.  In Subsection \ref{qz sec} we give the definition of $\delta$ in terms of $K$-theory with $\Q/\Z$ coefficients; we use this as the main definition as it seems closest to the applications we want to pursue.  In Subsection \ref{ext sec} we give a picture in the spirit of the $\text{Ext}$ picture of $K$-homology; this will not be used in the main body of the paper, and is included mainly for completeness, and as it is perhaps the simplest approach to the pairing.  In Subsection \ref{mc sec}, we relate our secondary pairing to the universal multicoefficient theorem of Dadarlat-Loring \cite{Dadarlat:1996aa}; this can be regarded as a third picture of the pairing, and seems the most useful of the three pictures in terms of relating the secondary pairing to the previous literature.

\subsection{Definition of the pairing using $\Q/\Z$ coefficients}\label{qz sec}

Let $G$ be a countable abelian group and $A$ be a $C^*$-algebra.  Let $j\in \{0,1\}$, and let $B_G$ be any separable $C^*$-algebra satisfying the UCT\footnote{There are many ways to produce such a $B_G$: see for example the proofs of \cite[Theorem 13.2.4]{Rordam:2000mz} or \cite[Corollary 23.10.3]{Blackadar:1998yq}.} equipped with a fixed choice of isomorphism $\phi_B:K_j(B_G)\to G$, and with $K_{j+1}(B_G)=0$.  We can then define a model for the \emph{$K$-theory of $A$ with coefficients in $G$} by 
\begin{equation}\label{k coeff def}
K_i(A;G):=K_{i+j}(A\otimes B_G).
\end{equation}
This does not depend on the choice of $B_G$: indeed, if $C_G$ is another separable UCT $C^*$-algebra equipped with an isomorphism $\phi_C:K_k(C_G)\to G$ and with $K_{k+1}(C_G)=0$, then the UCT short exact sequence of line \eqref{uct ses} (possibly coupled with taking a suspension) gives a unique $KK$-equivalence $\alpha \in KK^{j+k}(B_G,C_G)$ such that the diagram
$$
\xymatrix{ K_j(B_G) \ar[d]^-{\phi_B}_\cong \ar[r]^-\alpha_\cong & K_k(C_G) \ar[d]^-{\phi_C}_\cong \\ G \ar@{=}[r] & G }
$$
commutes; this $\alpha$ induces a $KK$-equivalence $\text{id}_A\otimes \alpha$ and so an isomorphism $K_*(A\otimes B_G)\cong K_*(A\otimes C_G)$ for all\footnote{If $A$ is not separable consider the maps on $K$-theory induced by $\text{id}_{A_i}\otimes \alpha$ as $A_i$ ranges over separable $C^*$-subalgebras of $A$, and take an inductive limit.} $C^*$-algebras $A$.  Note that $A\mapsto K_*(A;G)$ is a functor from the $KK$-category to the category of graded abelian groups: indeed, given $\alpha\in KK(A,C)$, we get $\alpha\otimes 1_{B_G}\in KK(A\otimes B_G,C\otimes B_G)$, and Kasparov product gives a map $K_*(A;G)\to K_*(C;G)$.

Let now $Q$ denote the universal UHF algebra, so $K_0(Q)\cong \Q$ and $K_1(Q)=0$, with the former isomorphism being canonically determined by the unique trace.  Let $\iota_Q:\C\to Q$ denote the unit inclusion, and let
\begin{equation}\label{cone qz}
C(\iota_Q):=\{f\in C_0([0,1],Q)\mid f(0)=0, f(1)\in \C1_Q\}
\end{equation}
denote its mapping cone.   We recall there is a short exact sequence 
\begin{equation}\label{qz ses}
0\to SQ\to C(\iota_Q)\to \C\to 0
\end{equation}
where $SQ$ is identified with $\{f\in C(\iota_Q)\mid f(1)=0\}$, and the map $C(\iota_Q)\to \C$ is $f\mapsto f(1)$.  The $K$-theory six-term exact sequence for this short exact sequence shows that $K_0(C(\iota_Q))=0$ and induces a commutative diagram
\begin{equation}\label{kci}
\xymatrix{ 0 \ar[r] & K_0(\C) \ar[r]^-{\text{exp}} \ar[d]^{\cong} & K_1(SQ) \ar[r] \ar[d]^{\cong} & K_1(C(\iota_Q)) \ar[d]^{\cong} \ar[r] &  0 \\
0 \ar[r] & \Z \ar[r] & \Q \ar[r] & \Q/\Z \ar[r] & 0  }
\end{equation}
where $\text{exp}$ is the usual exponential boundary map taking $[1]$ to the class of the function $[0,1]\to Q$, $t\mapsto e^{2\pi i t}$.  The isomorphism $K_0(\C)\to \Z$ is chosen to be the canonical one determined by the trace, and this choice (plus commutativity and unique divisibility of $\Q$) uniquely determines the other two vertical isomorphisms in line \eqref{kci}.  With these choices of isomorphisms, we may use $SQ$ and $C(\iota_Q)$ to define $K$-theory with coefficients in $\Q$ and $\Q/\Z$: precisely  
\begin{equation}\label{qz coeffs}
K_i(A;\Q/\Z):=K_{i+1}(A\otimes C(\iota_Q)) \quad \text{and}\quad K_i(A;\Q):=K_{i+1}(A\otimes SQ)
\end{equation}
One computes directly that
\begin{equation}\label{q coeffs}
K_i(A;\Q)\cong K_i(A)\otimes \Q
\end{equation}
(write $A\otimes Q$ as a direct limit of matrix algebras over $A$ and use continuity of $K$-theory, or use the K\"{u}nneth formula \cite{Schochet:1982aa}), and moreover that this isomorphism is compatible with the maps on $K_i(A;\Q)$ induced by elements of $KK(A,B)$.

Now, let $A$ be a separable $C^*$-algebra, and let $\alpha\in K^0(A)$ be an element in the kernel $k^0(A)$ of the pairing map $\gamma:K^i(A) \to \text{Hom}(K_i(A),\Z)$.  We then have a commutative diagram of long exact sequences
\begin{equation}\label{a to c}
\xymatrix{ \cdots \ar[r] & K_0(A;\Q) \ar[r] \ar[d] & K_0(A;\Q/\Z) \ar[r] \ar[d] & K_1(A) \ar[d] \ar[r] & K_1(A;\Q) \ar[r] \ar[d] & \cdots \\
\cdots \ar[r] & K_0(\C;\Q) \ar[r] & K_0(\C;\Q/\Z) \ar[r] & K_1(\C) \ar[r] & K_1(\C;\Q) \ar[r] & \cdots }
\end{equation}
where all vertical maps are induced by $\alpha$.  Using the isomorphism in line \eqref{q coeffs} and the fact that $\alpha$ is in $k^0(A)$, the first vertical map is zero.  Moreover, the the map $K_1(A)\to K_1(A;\Q)$ identifies with the map 
$$
K_1(A)\to K_1(A)\otimes \Q,\quad x\mapsto x\otimes 1_\Q,
$$
and thus its kernel is the torsion subgroup $t K_1(A)$.  Computing also the groups on the bottom row, the diagram in line \eqref{a to c} induces a commutative diagram
\begin{equation}\label{qz ver}
\xymatrix{ \cdots \ar[r] & K_0(A;\Q) \ar[r] \ar[d]^0 & K_0(A;\Q/\Z) \ar[r] \ar[d] & t K_1(A) \ar[d] \ar[r] & 0 \\
\cdots \ar[r] & \Q \ar[r] & \Q/\Z \ar[r] & 0 &   }.
\end{equation}
Hence from the first isomorphism theorem, the map $K_0(A;\Q/\Z)\to \Q/\Z$ induced by $\alpha$ induces a homomorphism
$$
\delta(\alpha):t K_1(A)\to \Q/\Z.
$$
Precisely analogously, if $\alpha$ is in $k^1(A)$, then it induces a homomorphism
$$
\delta(\alpha):t K_0(A)\to \Q/\Z.
$$
The assignment $\alpha\mapsto \delta(\alpha)$ is easily seen to be additive, so we get well-defined group homomorphisms
\begin{equation}\label{delta}
\delta: k^i(A)\to \text{Hom}( tK_{i+1}(A), \Q/\Z)
\end{equation}
for $i\in \{0,1\}$.   As the Kasparov product is continuous, it moreover descends to a map\footnote{We could of course adopt notation like ``$\delta^L_i$'' to signify exactly what is meant, but prefer to avoid clutter, and do not believe this will result in any confusion.}
\begin{equation}\label{deltal}
\delta: kl^i(A)\to \text{Hom}( tK_{i+1}(A), \Q/\Z)
\end{equation}
on the maximal Hausdorff quotient $kl^i(A)$ of $k^i(A)$ (equivalently, $kl^i(A)$ is the kernel of the map $\gamma$ in line \eqref{pair l map}).

\begin{definition}\label{sec pair def}
The \emph{secondary pairing} between $K$-homology and $K$-theory is defined to be any of the homomorphisms $\delta$ from lines \eqref{delta} or \eqref{deltal} above.  We will also sometimes (and equivalently) consider the associated $\Z$-bilinear pairing,
$$
k^i(A)\times \tor K_{i+1}(A)\to \Q/\Z,\quad (x,y)\mapsto \langle x,y\rangle_{\text{II}}
$$
and similarly with $kl^i(A)$ replacing $k^i(A)$.
\end{definition}

We record a naturality property of the secondary pairing for future use; it follows directly from associativity of the Kasparov product.

\begin{proposition}\label{nat}
The secondary pairing defined above is natural for maps induced by $KK$-classes: precisely, if $\beta \in KK^j(A,B)$, then $\beta$ induces maps $\beta^*:k^i(B)\to k^{i+j}(A)$ and $\beta_*:\tor K_i(A)\to \tor K_{i+j}(B)$ such that
$$
\langle \beta^*(x),y\rangle_{\text{II}}=\langle x,\beta_*(y)\rangle_{\text{II}}
$$
for all $x \in k^i(B)$ and $y\in \tor K_{i+1}(A)$. \qed
\end{proposition}

\begin{remark}\label{rz rem}
It will be useful for us to also discuss the secondary pairing using $\R/\Z$ coefficients rather than $\Q/\Z$, following Antonini-Azzali-Skandalis \cite{Antonini:2014aa}.  For this, let $B$ be any choice of II$_1$-factor, let $\iota_B:\C\to B$ be the unit inclusion, and let $C(\iota_B)$ be the mapping cone, i.e.\
$$
C(\iota_B):=\{f\in C_0([0,1],B)\mid f(0)=0, f(1)\in \C1_B\}.
$$ 
Let $A$ be a (separable) $C^*$-algebra in the UCT class.  Then Antonini-Azzali-Skandalis \cite[Section 3.4]{Antonini:2014aa} define the \emph{$K$-theory of $A$ with $\R/\Z$-coefficients} to be
$$
K_i(A;\R/\Z):= K_{i+1}(A\otimes C(\iota_B))
$$
and show that this does not depend on the choice of $B$; it is important here that $A$ is in the UCT class, as $B$ (and therefore $C(\iota_B)$) is probably not\footnote{A superficial reason for this is that it is not separable, but that could probably be got around by appropriate uses of direct limits.  A more serious problem is that the functor $A\mapsto K_*(A\otimes B)$ probably does not take short exact sequences of $C^*$-algebras to long exact sequences of $K$-theory groups, which seems to violate any reasonable form of the UCT one could hope to define for non-separable $C^*$-algebras.}.  Note also that $K_i(A;\R/\Z)$ is functorial for morphisms from the $KK$ category.  Indeed, let $A$ and $C$ be $C^*$-algebras in the UCT class, and write $C(\iota_B)$ as an (uncountable) inductive limit of separable $C^*$-algebras $B_i$.  Then given a class $\alpha\in KK(A,B)$, we may form the classes $\alpha\otimes \text{id}_{B_i}\in KK(A\otimes B_i,C\otimes B_i)$, take the maps $K_*(A\otimes B_i)\to K_*(C\otimes B_i)$ these induce on $K$-theory, and then take an inductive limit.

Also following Antonini-Azzali-Skandalis \cite[Section 3.3]{Antonini:2014aa}, we define $K_i(A;\R):=K_i(A\otimes B)$.  This again does not depend on the choice of $B$, and is functorial under morphisms from $KK$.  Moreover, as $A$ satisfies the UCT, the K\"{u}nneth formula gives $K_i(A;\R)\cong K_i(A)\otimes \R$ canonically, and the maps induced by $KK$-elements are compatible with these isomorphisms.

Now, fix a unital (necessarily trace-preserving) embedding $\iota:Q\to B$; such exists and is unique up to unitary equivalence, essentially as projections in $B$ are classified up to Murray-von Neumann equivalence by their traces.  We thus get a commutative diagram
\begin{equation}\label{q to b}
\xymatrix{ 0 \ar[r] & SQ\ar[r] \ar[d] & C(\iota_Q) \ar[d] \ar[r] & \C \ar@{=}[d] \ar[r] & 0 \\
0 \ar[r] & SB\ar[r] & C(\iota_B) \ar[r] & \C \ar[r] & 0 }
\end{equation}
with all vertical maps induced by $\iota$.  Let $\alpha$ be an element of $kl^0(A)$.  Applying the argument that gave rise to line \eqref{qz ver} and the commutative diagram in line \ref{q to b}, we get a commutative diagram
{\tiny $$
\xymatrix{ \cdots \ar[r] & K_0(A)\otimes \Q \ar'[d][dd]^-{0} \ar[rr] \ar[dr] & & K_0(A;\Q/\Z) \ar[dr]\ar'[d][dd] \ar[rr] &&  \tor K_1(A) \ar@{=}[dr] \ar'[d][dd]\ar[r] &0  \\
& \cdots \ar[r] & K_0(A)\otimes \R  \ar[dd]^(.7){0} \ar[rr] & & K_0(A;\R/\Z) \ar[dd] \ar[rr] && \tor K_1(A) \ar[dd]\ar[r] & 0 \\
\cdots  \ar[r]^-{\times n} & \Z \ar[dr] \ar'[r][rr] & &  \Q/\Z \ar'[r][rr] \ar[dr] & & 0 \ar@{=}[dr] &  \\
& \cdots  \ar[r]& \R \ar[rr] & & \R/\Z \ar[rr] &  & 0 &   }.
$$}
with all vertical maps induced by $\alpha$, and diagonal maps induced by the vertical maps in line \eqref{q to b}.  It follows from this commutative diagram that we may define the pairing 
$$
\tor K_1(A)\to \Q/\Z
$$
using the `back face' of this diagram and $\Q/\Z$ coefficients as in Definition \ref{sec pair def} just as well using the `front face' of the diagram and $\R/\Z$ coefficients getting a pairing
$$
\tor K_1(A)\to \R/\Z
$$
(note that this must actually take image in the torsion subgroup $\Q/\Z$ of $\R/\Z$), at least when $A$ satisfies the UCT.
\end{remark}

\subsection{Definition of the pairing using extensions}\label{ext sec}

We now give a second, equivalent, construction of the secondary pairing.  Let $A$ be a separable $C^*$-algebra, and let $\alpha$ be an element of $k^0(A)$ (recall this is the kernel of the natural map $\gamma:K^0(A)\to \text{Hom}(K_0(A),\Z)$).  Using the Brown-Douglas-Fillmore model of $K$-homology (see for example \cite[Chapter 2, and Sections 5.1-2 and 8.4]{Higson:2000bs}), we may represent $\alpha$ as a (semisplit) extension 
\begin{equation}\label{alpha ext}
0\to \mathcal{K}\to E\to SA\to 0
\end{equation}
of the suspension of $A$ by the compact operators $\mathcal{K}$.  Then the map $\gamma(\alpha):K_0(A)\to \Z$ given by the primary pairing of line \eqref{pair map} is precisely the boundary map in the six-term exact sequence
\begin{equation}\label{alpha 6t}
\xymatrix{ K_0(\mathcal{K})\ar[r] & K_0(E) \ar[r] & K_0(SA) \ar[d] \\ K_1(SA) \ar[u] & K_1(E) \ar[l] & K_1(\mathcal{K}) \ar[l] }
\end{equation}
when we identify $K_0(\mathcal{K})=\Z$ and $K_1(SA)=K_0(A)$ (see for example \cite[Exercise 7.7.2]{Higson:2000bs}).  As we are assuming $\alpha\in k^0(A)$, we have $\gamma(\alpha)=0$.  As also $K_1(\mathcal{K})=0$, the six-term exact sequence in line \eqref{alpha 6t} then gives rise to a short exact sequence
\begin{equation}\label{kap al}
0\to \Z\to K_0(E)\to K_1(A)\to 0.
\end{equation}
Now, consider the diagram
\begin{equation}\label{ext dash}
\xymatrix{ 0 \ar[r] & \Z \ar@{=}[d] \ar[r] & \ar@{-->}[d] K_0(E) \ar[r] & K_1(A) \ar[r] & 0 \\
0 \ar[r] & \Z \ar[r] & \Q \ar[r] & \Q/\Z \ar[r] & 0 }.
\end{equation}
As $\Q$ is injective, the central dashed arrow can be filled in (non-uniquely) so that the diagram commutes. The dashed arrow then induces a homomorphism $\widetilde{\delta_2(\alpha)}:K_1(A)\to \Q/\Z$.  The homomorphism $\widetilde{\delta_2(\alpha)}$ depends on the choice of extension $K_0(E)\to \Q$; however, we have the following lemma.

\begin{lemma}\label{d2 wd}
The restriction $\delta_2(\alpha):\tor K_1(A)\to \Q/\Z$ of $\widetilde{\delta_2(\alpha)}$ to the torsion subgroup of $A$ does not depend on the choice of extension $\phi:K_0(E)\to \Q$.
\end{lemma}

\begin{proof}
Let $\phi,\psi:K_0(E)\to \Q$ be two homomorphisms that fit in as the dashed arrow in line \eqref{ext dash} above.  Then the difference $\phi-\psi$ vanishes on the subgroup $\Z$, so descends to a homomorphism $\psi-\phi:K_1(A)\to \Q$.  This map must send the torsion subgroup of $K_1(A)$ to zero, whence the maps induced by $\phi$ and $\psi$ from $\tor K_1(A)$ to $\Q/\Z$ are the same.
\end{proof}

Lemma \ref{d2 wd}, plus the analogous construction starting with a class in $k^1(A)$, thus gives us homomorphisms 
\begin{equation}\label{d2 def}
\delta_2:k^i(A)\to \text{Hom}(\tor K_{i+1}(A),\Q/\Z).
\end{equation}
This is our second definition of the secondary pairing.  

\begin{remark}\label{ext tor rem}
Let us explain the algebra going into the definition of $\delta_2$ in a slightly different way, more in the spirit of the classical UCT; this will help connect to some of the older literature.  Let 
\begin{equation}\label{ext def 1}
\kappa:k^0(A)\to \text{Ext}(K_1(A),\Z)
\end{equation}
be the map sending a class $\alpha$ to the class of the associated extension as in line \eqref{kap al}.  This is a homomorphism (see for example \cite[Lemma 7.6.10]{Higson:2000bs}), and is in fact an isomorphism in the presence of the UCT.  On the other hand, as $\text{Ext}$ is contravariantly functorial in the first variable, there is a `restriction' homomorphism
\begin{equation}\label{ext def 2}
\text{Ext}(K_1(A),\Z)\to \text{Ext}(\tor K_1(A),\Z).
\end{equation}
Finally, note that for any abelian group, we may compute $\text{Ext}(H ,\Z)$ using the injective resolution $0\to \Z\to\Q\to\Q/\Z\to 0$ of $\Z$, which implies that 
$$
\text{Ext}(H,\Z)\cong \frac{\text{Hom}(H,\Q/\Z)}{\text{Im}(\text{Hom}(H,\Q))},
$$
where the right hand side means the quotient of $\text{Hom}(H,\Q/\Z)$ by the image of the natural map $\text{Hom}(H,\Q)\to \text{Hom}(H,\Q/\Z)$ (i.e.\ by the subgroup of homomorphisms that lift to $\Q$).  As $\tor K_1(A)$ is torsion, we thus get an isomorphism 
\begin{equation}\label{ext def 3}
\text{Ext}(\tor K_1(A),\Z)\stackrel{\cong}{\to} \text{Hom}(\tor K_1(A),\Q/\Z).
\end{equation}
We leave it to the reader to check that $\delta_2:k^0(A)\to \text{Hom}(\tor K_1(A),\Q/\Z)$ is the composition of the isomorphisms in lines \eqref{ext def 1}, \eqref{ext def 2}, and \eqref{ext def 3}: this is really just pushing definitions around.  Of course, one can also switch degrees in the above and define the $k^1(A)$ version of $\delta$ similarly.
\end{remark}

We now check that the pairing of line \eqref{d2 def} is the same as the secondary pairing $\delta$ from Definition \ref{sec pair def} above.

\begin{proposition}\label{d=d2}
For $i\in \{0,1\}$, the homomorphisms 
$$
\delta,\delta_2:k^i(A)\to \text{Hom}(\tor K_{i+1}(A),\Q/\Z)
$$
are the same.
\end{proposition}

\begin{proof}
For notational simplicity, we focus on the case of $\alpha\in k^0(A)$; the case of $k^1(A)$ is essentially the same.  Let $\delta(\alpha):\tor K_1(A)\to \Q/\Z$ be as in Definition \ref{sec pair def}.  Using injectivity of $\Q/\Z$, extend this (arbitrarily) to a homomorphism 
\begin{equation}\label{wtd}
\widetilde{\delta(\alpha)}:K_1(A)\to \Q/\Z.
\end{equation}

Let us represent $\alpha$ as a (semisplit) extension $0\to \mathcal{K}\to E\to SA\to 0$ as in line \eqref{alpha ext}, and let $0\to SQ\to C(\iota_Q)\to \C\to 0$ be mapping cone extension from line \eqref{qz ses}.  Taking the tensor product of these gives a square of exact sequences 
\begin{equation}\label{3x3}
\xymatrix{ & 0 \ar[d]& 0 \ar[d]& 0\ar[d] & \\
0 \ar[r]  & \mathcal{K}\otimes SQ \ar[r] \ar[d] & E\otimes SQ \ar[r] \ar[d] & SA\otimes SQ \ar[d] \ar[r] & 0 \\
0 \ar[r]  & \mathcal{K}\otimes C(\iota) \ar[r] \ar[d] & E\otimes C(\iota) \ar[r] \ar[d] & SA\otimes C(\iota) \ar[d]  \ar[r] & 0 \\
0 \ar[r] & \mathcal{K}\otimes \C \ar[r] \ar[d] & E\otimes \C \ar[r] \ar[d] & SA\otimes \C \ar[r] \ar[d] & 0\\
& 0 & 0 & 0 & }.
\end{equation}
Taking $K$-theory of part of this diagram, we consider the following diagram
\begin{equation}\label{bok}
\xymatrix{ & & & K_0(\mathcal{K}\otimes \C) \ar[d]  \\
& &K_0(SA\otimes SQ) \ar[r] \ar[d] & K_1(\mathcal{K}\otimes SQ) \ar[d]  \\
& K_0(E\otimes C(\iota_Q)) \ar[r] \ar[d] & K_0(SA\otimes C(\iota_Q)) \ar[r] \ar[d]  & K_1(\mathcal{K}\otimes C(\iota_Q)) \\
K_0(\mathcal{K}\otimes \C) \ar[r] & K_0(E\otimes \C) \ar[r]^-\pi \ar[d] & K_0(SA\otimes \C) \ar[ur]_-{\widetilde{\delta(\alpha)}} & \\
& K_1(E\otimes SQ) & & &}.
\end{equation}
Here all horizontal and vertical arrows are induced by those in line \eqref{3x3} either directly or as $K$-theory boundary maps, and the map $\widetilde{\delta(\alpha)}$ makes the diagram commute by definition.

We claim that if $x\in K_0(E\otimes \C)$ is a torsion element, then $(\widetilde{\delta(\alpha)}\circ \pi)(x)=0$.  Indeed, if $x$ is torsion it maps to zero under the vertical map $K_0(E\otimes \C)\to K_1(E\otimes SQ)\cong K_0(E)\otimes \Q$ and therefore comes from an element $\widetilde{x}\in K_0(E\otimes C(\iota_Q))$ by exactness.  However, the composition 
$$
K_0(E\otimes C(\iota_Q)) \to K_0(SA\otimes C(\iota_Q)) \to K_1(\mathcal{K}\otimes C(\iota_Q)) 
$$
sends $\widetilde{x}$ to zero by exactness.  This implies that $(\widetilde{\delta(\alpha)}\circ \pi)(x)=0$ by commutativity of the diagram in line \eqref{bok}.

Now, let us simplify diagram \eqref{bok} to the diagram below by removing some data we no longer need, and computing the groups in the right hand column:
\begin{equation}\label{bok2}
\xymatrix{ & & & \Z \ar[d]  \\
& & & \Q \ar[d]  \\
& &  & \Q/\Z  \\
K_0(\mathcal{K})\ar[r] & K_0(E) \ar[r]^-\pi  & K_1(A) \ar[ur]_-{\widetilde{\delta(\alpha)}} & }.
\end{equation}
It follows from the claim that the map 
$$
\widetilde{\delta(\alpha)}\circ \pi:K_0(E)\to K_1(\mathcal{K}\otimes C(\iota_Q))=\Q/\Z
$$
factors in the form
$$
\xymatrix{ K_0(E) \ar[rr]^-{\text{id}\otimes 1_\Q} && K_0(E)\otimes \Q \ar[rr]^-{(\widetilde{\delta(\alpha)}\circ \pi)\otimes \rho } & & \Q/\Z},
$$
where $\rho:\Q\to \Q/\Z$ is the canonical quotient; note that the second of the maps above is naturally a $\Q$-module map.  We add this data to diagram \eqref{bok2} as follows
\begin{equation}\label{bok3}
\xymatrix{ & & & & \Z \ar[d] & \\
& & & & \Q \ar[d] & \\
& & K_0(E)\otimes \Q \ar@{-->}[urr] \ar[rr]^-{(\widetilde{\delta(\alpha)}\circ\pi)\otimes \rho} & & \Q/\Z \\
K_0(\mathcal{K}) \ar@{=}[uuurrrr] \ar[r] & K_0(E) \ar[r]^-\pi \ar[ur] & K_1(A) \ar[urr]_-{\widetilde{\delta(\alpha)}} & & }
\end{equation}
where the equality sign is the canonical identification, and the dashed arrow will be explained next. Indeed, $K_0(E)\otimes \Q$ is a $\Q$-vector space (in particular, it is free as a $\Q$-module and all its $\Q$-submodules are free and complemented) and $(\widetilde{\delta(\alpha)}\circ\pi)\otimes \rho$ is a $\Q$-module map, from which it follows that we can fill in the dashed arrow in line \eqref{bok3} with a $\Q$-module map so that the diagram commutes.  

Removing the no-longer-necessary data from the middle of the diagram, we now have the following commutative diagram:
$$
\xymatrix{ & & & \Z \ar[d]  \\
& & & \Q \ar[d]  \\
& &  & \Q/\Z  \\
K_0(\mathcal{K})\ar[r] \ar@{=}[uuurrr]& K_0(E) \ar[uurr] \ar[r]^-\pi  & K_1(A) \ar[ur]_-{\widetilde{\delta(\alpha)}} & }.
$$
According to the definition of $\delta_2(\alpha)$ (compare line \eqref{ext dash} above), this implies that $\delta_2(\alpha)$ agrees with the restriction of $\widetilde{\delta(\alpha)}$ to $\tor K_1(A)$, and we are done.
\end{proof}

\subsection{Relationship to the universal multicoefficient theorem}\label{mc sec}

We now relate our secondary pairing to the universal multicoefficient theorem (UMCT) of Dadarlat-Loring \cite{Dadarlat:1996aa}, and also to Dadarlat's work \cite{Dadarlat:2005aa} on the topology on $KK$-theory and $KL$-theory.  We then use the UMCT to show that if $A$ satisfies the UCT, then the secondary pairing maps
$$
\delta:kl^i(A)\to \text{Hom}(t K_{i+1}(A),\Q/\Z)
$$
are homeomorphic isomorphisms for $i\in \{0,1\}$.

We need to recall some background on $K$-theory with finite coefficients and total $K$-theory.   Analogously to line \eqref{cone qz} above, let $\iota_n:\C\to M_n(\C)$ be the unit inclusion, and let $I_n:=C(\iota_n)$ be its mapping cone, i.e.\
\begin{equation}\label{dd alg}
I_n:=\{f\in C([0,1],M_n(\C))\mid f(0)=0,~f(1)\in \C1_{M_{n}(\C)}\}.
\end{equation}
The $C^*$-algebra $I_n$ is usually called a \emph{dimension drop} $C^*$-algebra.  Taking the $K$-theory six-term exact sequence associated to the short exact sequence 
\begin{equation}\label{dd ses}
0\to SM_n(\C)\to I_n\to \C\to 0,
\end{equation}
one computes\footnote{See for example \cite[Section 13.1]{Rordam:2000mz}, or use general facts on mapping cones as in for example \cite[Exercise 6.M]{Wegge-Olsen:1993kx}.} that $K_1(I_n)\cong \Z/n$ and $K_0(I_n)=0$.  Analogously to the discussion around line \eqref{kci} above, the isomorphism $K_1(I_n)\cong \Z/n$ is canonically determined by the short exact sequence in line \eqref{dd ses} and the $K$-theory six-term exact sequence, so we may fix this isomorphism and define $K_i(A;\Z/n):=K_{i+1}(A\otimes I_n)$ for any $C^*$-algebra $A$.  For notational convenience, we also set $I_1:=\C$.

We now define the category $\Lambda$ underlying the universal multicoefficient theorem.  

\begin{definition}\label{lambda def}
Write $\N$ for the positive integers $1,2,3,...$, and write the elements of $\Z/2$ as $0,1$.  We let $\Lambda$ denote the (additive) category with objects $\Z/2\times \N$, morphisms from $(i,n)$ to $(j,m)$ given by $KK^{i+j}(I_n,I_m)$, and composition given by Kasparov product.  
\end{definition}

\begin{examples}\label{lambda ex}
We compute the morphism groups $KK^j(I_n,I_m)$ underlying $\Lambda$ concretely, and exhibit generators; this will be important for some computations later in the section.  This is based on the notation and material in \cite{Dadarlat:1996aa}: it is known to experts but we do not know a reference for it in the literature, so provide arguments.  

Each of the points below computes some special cases of the morphism groups $KK^j(I_n,I_m)$ as abstract groups: these computations follow from the UCT exact sequence as in line \eqref{uct ses}, and we leave them to the reader.   We then exhibit an explicit generator in each case: in fact, the groups $KK^j(I_n,I_m)$ all turn out to be cyclic, so to exhibit generators it suffices to exhibit an element of $KK^j(I_n,I_m)$ whose order is (at least) that of the group.  Moreover, to bound the order of an element $\alpha\in KK(A,B)$ below, it suffices to compute the order of its image in some group $\text{Hom}(K_*(A\otimes C),K_*(B\otimes C))$; this is generally what we will do. 
\begin{enumerate}[(i)]
\item \label{kap mn n} For $n,m\geq 2$, the group $KK(I_n,I_m)$ is isomorphic to $\Z/\text{gcd}(m:n)$.  

First, in the special case $KK(I_n,I_{mn})$ it is generated by the class of the homomorphism $\kappa_{mn,n}:I_n\to I_{mn}$ induced by the canonical unital inclusion $M_n\to M_{mn}$. To see that $\kappa_{mn,n}\in KK(I_{mn},I_n)$ is a generator, consider first the commutative diagram 
\begin{equation}\label{kap mn n diag}
\xymatrix{ 0 \ar[r] & SM_n \ar[r] \ar[d]^-{\kappa_{mn,n}} & I_n\ar[r] \ar[d]^-{\kappa_{mn,n}} & \C \ar[r] \ar@{=}[d] & 0 \\
0 \ar[r] & SM_{mn} \ar[r] & I_{mn}\ar[r] & \C \ar[r] & 0 }.
\end{equation}
Note that the map induced on $K$-theory by the restriction of $\kappa_{mn,n}$ to $SM_n$ is multiplication by $m$.  Considering the $K$-theory groups of the diagram in line \eqref{kap mn n diag} that $\kappa_{mn.n}$ induces the map 
$$
\Z/n\cong K_1(I_n)\to K_1(I_{mn})\cong \Z/(mn)
$$
defined by sending $[1]$ to $[m]$.  This has order $n$ in $\text{Hom}(\Z/n,\Z/(mn))$, so $\kappa_{mn,n}$ generates $KK(I_n,I_{mn})\cong \Z/n$.  

Similarly, In the special case $KK(I_{mn},I_n)$, it is generated by the class of the natural inclusion $\kappa_{n,mn}:I_{mn}\to M_m(I_n)$. To see this, consider the diagram
\begin{equation}\label{kap n mn diag}
\xymatrix{ 0 \ar[r] & SM_{mn} \ar[r] \ar[d]^-{\kappa_{n,mn}} & I_n\ar[r] \ar[d]^-{\kappa_{n,mn}} & \C \ar[r] \ar[d] & 0 \\
0 \ar[r] & M_m(SM_n) \ar[r] & I_{mn}\ar[r] & M_m(\C) \ar[r] & 0 }.
\end{equation}
Note that the map induced on $K$-theory by the restriction of $\kappa_{mn,n}$ to $SM_{mn}$ is the identity (having canonically identified the groups with $\Z$).  Hence $\kappa_{n,mn}$ induces the natural surjection 
$$
\Z/(mn)\cong K_1(I_{mn})\to K_1(I_{n})\cong \Z/n
$$
defined by sending $[1]$ to $[1]$.  This has order $n$ in $\text{Hom}(\Z/(mn),\Z/n)$, so $\kappa_{n,mn}$ generates $KK(I_{mn},I_n)\cong \Z/n$. 

In general, the group $KK(I_n,I_m)$ is generated by $\kappa_{m,\text{gcd}(m:n)}\circ \kappa_{\text{gcd}(m:n),n}$.  Indeed, the computations above combine to show that this map induces the map
$$
\Z/n\cong K_0(I_n)\to K_0(I_m)\cong \Z/m
$$
sending $[1]$ to $[m/\text{gcd}(m:n)]$, so has the right order.

\item \label{rhom} For $m\geq 2$, the group $KK(I_1,I_m)$ is isomorphic to $\Z/m$, and is generated by the class of the inclusion homomorphism $\rho_m:SM_m\to I_m$ appearing in line \eqref{dd ses}.  To see that this generates, the six-term exact sequence in $K$-theory associated to the short exact sequence in line \eqref{dd ses} shows that the map 
$$
\rho_m:\Z\cong K_1(SM_n)\to K_1(I_m)\cong \Z/m
$$
it induces in the usual quotient map sending $1$ to $[1]$.  This has order $m$, whence generates $KK(I_1,I_m)\cong \Z/m$.
\item \label{betam} For $m\geq 2$, the group $KK^1(I_m,I_1)$ is isomorphic to $\Z/m$, and is generated by the class of the homomorphism $\beta_m:I_m\to \C$ appearing in line \eqref{dd ses} above.  To see that it generates, let $A=SI_m$.  Tensoring $A$ by the short exact sequence in line \eqref{dd ses} (with $m=n$) and taking $K$-theory gives a six-term exact sequence 
\begin{equation}\label{betam les}
\xymatrix{ K_0(SM_m(A)) \ar[r]^-{\rho_m} & K_0(A\otimes I_m) \ar[r]^-{\beta_m} & K_0(A) \ar[d] \\
K_1(A) \ar[u] & K_1(A\otimes I_m) \ar[l]^-{\beta_m} & K_1(SM_m(A)) \ar[l]^-{\rho_m} }.
\end{equation}
Note that: $K_0(A)\cong \Z/m$ and $K_1(A)=0$; having identified $K_i(SM_m(A))$ with $K_{i+1}(A)$, the boundary maps induce multiplication by $m$; and that $K_i(A\otimes I_m)\cong \Z/m$ for $i\in \{0,1\}$ (e.g.\ from the K\"{u}nneth formula \cite{Schochet:1982aa}).  Hence line \eqref{betam les} above simplifies to 
\begin{equation}\label{rhom les}
\xymatrix{ 0 \ar[r]^-{\rho_m} & \Z/m \ar[r]^-{\beta_m} & \Z/m \ar[d]^-{\times m} \\
0 \ar[u] & \Z/m \ar[l]^-{\beta_m}  & \Z/m \ar[l]^-{\rho_m}}.
\end{equation}
As the boundary map is zero, the map 
$$
\Z/m\cong K_0(A\otimes I_m)\to K_0(A)\cong \Z/m
$$
induced by $\beta_m$ is an isomorphism.  Hence $\beta_m$ has order at least $m$ in $KK^1(I_1,I_m)$, so generates.
\item For $m,n\geq 2$, the group $KK^1(I_n,I_m)$ is isomorphic to $\Z/\text{gcd}(n;m)$.  It is generated by the composition $\rho_m\circ \beta_n$.  To see that this generates, we consider $A=SI_m$ as above, but now tensor the short exact sequence in line \eqref{dd ses} for arbitrary $n\geq 2$ by $A$.  The analogue of the diagram in line \eqref{betam les} then simplifies to 
$$
\xymatrix{ 0 \ar[r] & \Z/m\otimes \Z/n \ar[r]^-{\beta_n} & \Z/m \ar[d]^-{\times n} \\
0 \ar[u] & \Z/m\otimes \Z/n \ar[l]^-{\beta_n}  & \Z/m \ar[l]}.
$$
As $\Z/m\otimes \Z/n\cong \Z/\text{gcd}(n:m)$ it follows that $\beta_n$ induces the canonical inclusion
$$
\Z/\text{gcd}(n:m)\cong K_1(A\otimes I_m)\to K_0(A)\cong \Z/m
$$
sending $[1]$ to $[m/\text{gcd}(m:n)]$.  On the other hand, the diagram in line \eqref{rhom les} shows that for this $A$, the map 
$$
\Z/m\cong K_0(A)\to K_1(A\otimes I_m)\cong \Z/m
$$
induced by $\rho_m$ is an isomorphism.  It follows that the composition $\rho_m\circ \beta_n$ induces a map 
$$
\Z/\text{gcd}(m:n)\cong K_1(A\otimes I_n)\to K_0(A)\to K_0(A\otimes I_m)\cong \Z/m
$$
of order $\text{gcd}(n:m)$, and therefore $\rho_m\circ \beta_n$ generates $KK^1(I_n,I_m)\cong \Z/\text{gcd}(n:m)$.
\item The group $KK(I_1,I_1)$ is isomorphic to $\Z$ generated by the class of the identity morphism. 
\item For any $m\geq 2$, $KK(I_m,I_1)=KK^1(I_1,I_m)=KK^1(I_1,I_1)=0$.
\end{enumerate}
\end{examples}

We will need two technical lemmas about the morphisms from part \eqref{kap mn n} above.  For the statement of the first, for an abelian group $G$, we recall the usual Tor group of homological algebra, which can be concretely defined by 
\begin{equation}\label{tor zn}
\text{Tor}(\Z/n,G):=\{g\in G\mid ng=0\}.
\end{equation}

\begin{lemma}\label{rhom coker}
Let $A$ be a $C^*$-algebra.  Tensoring the short exact sequence of line \eqref{dd ses} by $A$ and taking $K$-theory gives
\begin{equation}\label{base les}
\xymatrix{ \cdots \ar[r] & K_1(SM_n(A)) \ar[r]^-{\rho_n} & K_1(A\otimes I_n) \ar[r] & K_1(A) \ar[r] & K_0(SM_n(A)) \ar[r] & \cdots }.
\end{equation}
The cokernel of the map $\rho_n$\footnote{Technically, the map on $K$-theory induced by the map $\rho_n\otimes \text{id}_A\in KK(I_1\otimes A,I_n\otimes A)$, but we will elide this.} is canonically isomorphic to $\text{Tor}(\Z/n,K_1(A))$.
\end{lemma}

\begin{proof}
Analogously to the computations showing that $K_1(I_n)\cong \Z/n$ cited above\footnote{i.e.\ based on \cite[Section 13.1]{Rordam:2000mz}, or using general facts on mapping cones as in for example \cite[Exercise 6.M]{Wegge-Olsen:1993kx} and that $C(\iota_n)\otimes A$ is the mapping cone of the diagonal inclusion $A\to M_n(A)$.}, one computes that having made the identification $K_0(SM_n(A))=K_1(A)$, the boundary map $K_1(A)\to K_0(SM_n(A))$ above is given by multiplication by $n$.  The result follows from exactness.
\end{proof}

\begin{lemma}\label{kap lem}
The morphisms $\kappa_{mn,n}$ and $\kappa_{n,mn}$ of Examples \ref{lambda ex} part \eqref{kap mn n} canonically induce maps on the cokernels of $\rho_{n}$ and $\rho_{mn}$ respectively, i.e.\ according to Lemma \ref{rhom coker} we have 
$$
\kappa_{mn,n}:\text{Tor}(\Z/n,K_1(A))\to \text{Tor}(\Z/(mn),K_1(A))
$$
and 
$$
\kappa_{n,mn}:\text{Tor}(\Z/(mn),K_1(A))\to \text{Tor}(\Z/n,K_1(A)).
$$
Moreover, the first of these is the canonical inclusion, and the second is given by multiplication by $m$.
\end{lemma}

\begin{proof}
For $\kappa_{mn,n}$ we tensor the commutative diagram of short exact sequences in line \eqref{kap mn n diag} by $A$, and take $K$-theory to get a commutative diagram 
$$
\xymatrix{ \cdots \ar[r] & K_0(A) \ar[r]^-{\rho_m} \ar[d]^-{\times m} & K_0(A;\Z/n) \ar[r]^-{\beta_m} \ar[d]^-{\kappa_{mn,n}} & K_1(A) \ar@{=}[d] \ar[r]^{\times n} & K_1(A)\ar[r] \ar[d]^{\times m} & \cdots \\
\cdots \ar[r] & K_0(A) \ar[r]^-{\rho_{mn}} & K_0(A;\Z/(mn)) \ar[r]^-{\beta_{mn}} & K_1(A) \ar[r]^{\times mn} & K_1(A)\ar[r] & \cdots }.
$$
The claimed statement follows directly from this.  Similarly, for $\kappa_{n,mn}$, we tensor the diagram in line \eqref{kap n mn diag} by $A$ and take $K$-theory to get a commutative diagram 
$$
\xymatrix{ \cdots \ar[r] & K_0(A) \ar[r]^-{\rho_{mn}} \ar@{=}[d] & K_0(A;\Z/(mn)) \ar[r]^-{\beta_{mn}} \ar[d]^-{\kappa_{n,mn}} & K_1(A) \ar[d]^-{\times m} \ar[r]^{\times mn} & K_1(A)\ar[r] \ar@{=}[d] & \cdots \\
\cdots \ar[r] & K_0(A) \ar[r]^-{\rho_m} & K_0(A;\Z/n) \ar[r]^-{\beta_m} & K_1(A) \ar[r]^{\times n} & K_1(A)\ar[r] & \cdots }
$$
from which the claimed result follows.
\end{proof}

Going back now to generalities with $\Lambda$ as in Definition \ref{lambda def}, a \emph{$\Lambda$-module} is a covariant functor $\underline{G}$ from $\Lambda$ to the category of abelian groups.  A \emph{morphism} between $\Lambda$-modules is a natural transformation between such functors.   We write $\text{Hom}_{\Lambda}(\underline{G},\underline{H})$ for the abelian group of all morphisms between $\Lambda$-modules $\underline{G}$ and $\underline{H}$.

Let us describe this more concretely, mainly to establish notation that we will use later.  A $\Lambda$-module consists of a collection $\underline{G}=(G_{i,n})_{i\in \Z/2,n\in \N}$ of abelian groups together with homomorphisms 
\begin{equation}\label{lam mor}
KK^j(I_n,I_m)\to \text{Hom}(G_{i,n},G_{i+j,m})
\end{equation}
for each $i,j\in \Z/2$ and $n,m\in \N$ that are compatible with compositions in the $KK$-category.   A morphism in $\text{Hom}_{\Lambda}(\underline{G},\underline{H})$ consists of a collection $\phi=(\phi_{i,n}:G_{i,n}\to H_{i,n})_{i\in \Z/2,n\in \N}$ of group homomorphisms that are compatible with the images of the morphisms in $\Lambda$ under the maps in line \eqref{lam mor}.   We equip $\text{Hom}_{\Lambda}(\underline{G},\underline{H})$ with the topology of pointwise convergence, i.e.\ the topology it inherits from the natural inclusion 
$$
\text{Hom}_{\Lambda}(\underline{G},\underline{H})\subseteq \prod_{i\in \Z/2,n\in \N} \text{Hom}(G_{i,n},H_{i,n})
$$
where the $G_{i,n}$ and $H_{i,n}$ have the discrete topology, each $\text{Hom}(G_{i,n},H_{i,n})$ has the topology of pointwise convergence, and the product has the product topology.   

The motivating examples of $\Lambda$-modules are \emph{total $K$-groups}\footnote{The term ``total $K$-group'' is a bit of a misnomer: it is a $\Lambda$-module, not a group.  One could also package the same information into a single $\Z/2\times \N$-graded group, however as in \cite[Section 6]{Dadarlat:2012aa}.} of $C^*$-algebras: if $A$ is a $C^*$-algebra, we write $\underline{K}(A)$ for the $\Lambda$-module with $(i,n)^\text{th}$ group given by $K_i(A)$ for $n=1$, and $K_i(A;\Z/n)=K_{i+1}(A\otimes I_n)$ for $n\geq 2$.  The action of $\Lambda$ as in line \eqref{lam mor} is induced by Kasparov product.

\begin{remark}\label{tot kh}
We will not need it until later in the paper, but less us also note that one can similarly define total $K$-homology.  Indeed, if $A$ is a separable $C^*$-algebra, we may also define the \emph{$K$-homology of $A$ with coefficients in $\Z/n$} to be $K^i(A;\Z/n):=KK^{i+1}(A,I_n)$.  The collection of all $K$-homology groups of a $C^*$-algebra with all cyclic coefficients can be similarly made into a $\Lambda$-module, which we call the \emph{total $K$-homology} and denote by $\overline{K}(A)$.  If $A=C_0(X)$ for a locally compact space $X$, we instead write $\underline{K}(X)$ for total $K$-homology.
\end{remark}

Note that Kasparov product induces a canonical homomorphism 
\begin{equation}\label{under gamma}
\underline{\gamma}:KK(A,B)\to \text{Hom}_{\Lambda}(\underline{K}(A),\underline{K}(B)).
\end{equation}
This is compatible with the usual pairing 
$$
\gamma:KK(A,B)\to \text{Hom}(K_*(A),K_*(B))
$$
in the following sense.  Indeed, for a $\Lambda$-module $\underline{G}$, let $G_*$ denote the graded abelian group with $G_i:=G_{i,1}$ and define a \emph{forgetful map}
\begin{equation}\label{flambda}
f_\Lambda:\text{Hom}_{\Lambda}(\underline{G},\underline{H})\to \text{Hom}(G_{*},H_{*}),\quad (\phi_{i,n})_{i\in \Z/2,n\in \N}\mapsto (\phi_{i,1})_{i\in \Z/2}.
\end{equation}
The fact that $\gamma$ and $\underline{\gamma}$ are compatible just means that
\begin{equation}\label{gfl}
f_\Lambda\circ \underline{\gamma}=\gamma.
\end{equation}

The following lemma is the key technical ingredient needed for our main result in this section.   The proof is a little long\footnote{It is even longer that it looks: the main reason we went over Examples \ref{lambda ex} in detail and proved Lemma \ref{kap lem} was that those computations are needed to prove this lemma.  It would be interesting to have an approach that is more conceptual, and does not need the explicit descriptions of generators of $\Lambda$.}, but not deep.  

\begin{lemma}\label{annoying algebra}
Let $A$ be a separable $C^*$-algebra, and let $f_\Lambda$ be as in line \eqref{flambda} for $\underline{G}=\underline{K}(A)$ and $\underline{H}=\underline{K}(\C)$.    Then there is an injective homomorphism 
$$
\epsilon:\text{ker}(\Lambda_f)\to \prod_{m\geq 2} \text{Hom}(\text{Tor}(\Z/m,K_1(A)),\Z/m).
$$
The image of $\epsilon$ consists exactly of those tuples $(\psi_m)_{m\geq 2}$ such that for all $m,n\geq 2$, the following diagram commutes
\begin{equation}\label{kap tor}
\xymatrix{ \text{Tor}(\Z/m,K_1(A))\ar[r]^-{\kappa_{mn,m}} \ar[d]^-{\psi_{m}} & \text{Tor}(\Z/(mn),K_1(A))\ar[r]^-{\kappa_{n,mn}} \ar[d]^-{\psi_{mn}} & \text{Tor}(\Z/n,K_1(A)) \ar[d]^-{\psi_{n}}  \\
 \Z/m \ar[r]^-{[1]\mapsto [n]} & \Z/(mn) \ar[r]^-{[1]\mapsto [1]} & \Z/n }
\end{equation}
where the morphisms on the top row are those from Lemma \ref{kap lem}.
\end{lemma}

\begin{proof}
Write $\phi=(\phi_{i,n})_{i\in \Z/2,n\in \N}$ for an element of the kernel of $f_\Lambda$, and note that $\phi_{i,n}=0$ if $i=1$ or $n=0$ (or both), as $\underline{K}(\C)$ vanishes at these indices.  Consider the six-term exact sequences in $K$-theory induced by the short exact sequence in line \eqref{dd ses} and its tensor product with $A$.  As all maps on $K$-theory associated to these short exact sequences are induced by elements of $\Lambda$, we get a commutative diagram
$$
\xymatrix{ \cdots \ar[r] & K_0(A) \ar[r]^-{\rho_m} \ar[d]^-{\phi_{0,1}} & K_0(A;\Z/n) \ar[r]^-{\beta_m} \ar[d]^-{\phi_{0,n}} & K_1(A) \ar[d]^-{\phi_{1,1}} \ar[r] & K_1(A)\ar[r] \ar[d]^-{\phi_{1,1}} & \cdots \\
\cdots \ar[r] & K_0(\C) \ar[r]^-{\rho_m} & K_0(\C;\Z/n) \ar[r]^-{\beta_m} & K_1(\C) \ar[r] & K_1(\C)\ar[r] & \cdots }
$$
(here the maps $\rho_m$ and $\beta_m$ are induced by the $KK$-elements in Examples \ref{lambda ex} parts \eqref{rhom} and \eqref{betam}, and we draw the six-term exact sequences `unrolled' to fit better on the page).  Inserting the known values of the groups on the bottom row, and the fact that the morphism $\phi_{0,1}$ is zero, this becomes
$$
\xymatrix{ \cdots \ar[r] & K_0(A) \ar[r]^-{\rho_m} \ar[d]^-{0} & K_0(A;\Z/n) \ar[r]^-{\beta_m} \ar[d]^-{\phi_{0,n}} & K_1(A) \ar[d]^-{0} \ar[r] & K_1(A)\ar[r] \ar[d]^-{0} & \cdots \\
\cdots \ar[r] & \Z \ar[r] & \Z/n \ar[r] & 0 \ar[r] & 0  \ar[r] & \cdots }.
$$
Using exactness, $\phi_{0,n}$ uniquely determines, and is uniquely determined by, a map from the cokernel of $\rho_n$ to $\Z/n$, i.e.\ according to Lemma \ref{rhom coker}, a map
$$
\psi_n:\text{Tor}(\Z/n,K_1(A))\to \Z/n.
$$
These are the maps from the statement.  As $\phi_{0,n}$ are the only non-zero homomorphisms amongst those constituting $\phi$, it is clear that the map $\epsilon:\phi\mapsto (\psi_m)_{m\geq 2}$ is an injective homomorphism.

It remains to check the surjectivity statement and do the claimed computations with the $\kappa$ morphisms.  Let then $(\psi_m)_{m\geq 2}$ be a family such that the diagram in line \eqref{kap tor} commutes, and let $\phi_{0,n}:K_0(A;\Z/n)\to K_0(\C;\Z/n)$ be the maps defined as the composition 
$$
\xymatrix{ K_0(A;\Z/m) \ar[r]^-{\beta_m} & \text{Tor}(\Z/m,K_1(A)) \ar[r]^-{\psi_m} & \Z/m=K_0(\C;\Z/m) }
$$
(we are abusing notation: the map labeled ``$\beta_m$'' is as usual the map induced on $K$-theory by $\beta_m$, and then corestricted to its image).  We must show that the maps $\phi_{0,n}$ are compatible with the generators of $\Lambda$ described in Examples \ref{lambda ex}.  Indeed, compatibility with $\rho_m$ and $\beta_{m,n}$ (which factors through $\rho_m$) follows from the fact that $\phi_{0,n}$ descends to the cokernel $\text{Tor}(\Z/m,K_1(A))$ of the map induced on $K$-theory by $\rho_m$.  

It suffices now to check compatibility with the maps $\kappa_{mn,n}$ and $\kappa_{n,mn}$ from part \eqref{kap mn n}.  For $\kappa_{mn,n}$ this from Lemma \ref{kap lem} and as we computed in Examples \ref{lambda ex} part \eqref{kap mn n} that the map $K_0(\C;\Z/n)\to K_0(\C;\Z/(mn))$ induced by $\kappa_{mn,n}$ is the canonical inclusion sending $[1]$ to $[m]$.  Similarly, for $\kappa_{n,mn}$, this follows from Lemma \ref{kap lem} and our computation from Examples \ref{lambda ex} part \eqref{kap mn n} that the map $K_0(\C;\Z/(mn))\to K_0(\C;\Z/n)$ induced by $\kappa_{n,mn}$ is the canonical surjection sending $[1]$ to $[1]$.
\end{proof}

We are now ready to state the main result of this section.  For the proof, it will be convenient to write $J$ for the directed set 
\begin{equation}\label{j set}
J:=\{n\in \N\mid n\geq 2\}
\end{equation}
where the order is not the usual one, but instead we define ``$n\leq m$'' to mean ``$n$ divides $m$''.  

\begin{theorem}\label{totk}
For any separable $C^*$-algebra $A$ there is a commutative diagram of continuous maps of topological abelian groups
\begin{equation}\label{definitive}
\xymatrix{ 0 \ar[r] & k^0(A) \ar[r] \ar[d]^-\delta & K^0(A) \ar[d]^-{\underline{\gamma}} \ar[r]^-\gamma & \text{Hom}(K_0(A),\Z) \ar@{=}[d] & \\
0\ar[r] & \text{Hom}(\tor K_1(A),\Q/\Z)\ar[r] & \text{Hom}_{\Lambda}(\underline{K}(A),\underline{K}(\C)) \ar[r]^-{f_\Lambda} & \text{Hom}(K_0(A),\Z) \ar[r] &  0 }
\end{equation}
where $\delta$ is as in Definition \ref{sec pair def}, $\gamma$ is as in line \eqref{gamma pair}, $\underline{\gamma}$ is as in line \eqref{under gamma}, and $f_\Lambda$ is as in line \eqref{flambda}.  Moreover, the rows are exact, and the whole diagram is natural for morphisms in the $KK$-category.
\end{theorem}

\begin{remark}
There is also an `odd degree' version of Theorem \ref{totk} that we will not need, so do not state explicitly.  It can be deduced by replacing $A$ with its suspension.
\end{remark}

\begin{proof}[Proof of Theorem \ref{totk}]
The maps $\gamma$ and $\underline{\gamma}$ are continuous by continuity of the Kasparov product \cite[Theorem 3.5]{Dadarlat:2005aa}.  The map $f_\Lambda$ is continuous as the topologies on its domain and codomain are both topologies of pointwise convergence, and it is surjective by \cite[Proposition 2.15]{Carrion:2020aa}\footnote{The cited result is more general as it allows for an arbitrary separable $C^*$-algebra in place of $B$, and actually gives a split surjection.  The necessary surjectivity claim can be proved slightly more simply in the current context, but the algebra is still a little bit fiddly: this can probably be attributed to the fact that while $f_\Lambda$ splits, it does not do so naturally.}.   The right hand square in line \eqref{definitive} commutes by line \eqref{gfl}.  

It remains to construct a homomorphism 
$$
\text{Hom}(\tor K_*(A),\Q/\Z) \to \text{Hom}_{\Lambda}(\underline{K}(A),\underline{K}(\C))
$$
and show that it is continuous, and that the left hand square commutes.  We will actually construct a map from the kernel of $\Lambda_f$ to $\text{Hom}(\tor K_*(A),\Q/\Z)$, and show that it is a bijective homeomorphism.  

Let $\phi\in \text{Hom}_{\Lambda}(\underline{K}(A),\underline{K}(\C))$ be an element such that $\Lambda_f(\phi):K_0(A)\to \Z$ is zero.  For each $n\geq 2$, let 
$$
\psi_n:\text{Tor}(\Z/n,K_1(A))\to \Z/n
$$
be the homomorphism guaranteed by Lemma \ref{annoying algebra}.  Using commutativity of the first square in diagram in line \eqref{kap tor}, our maps $\psi_n$ thus fit into a commutative diagram
\begin{equation}\label{div}
\xymatrix{ \text{Tor}(\Z/n,K_1(A)) \ar[r]^-{\psi_n} \ar[d]_-{\kappa_{mn,n}} & \Z/n \ar[d]^-{\kappa_{mn,n}} \\
\text{Tor}(\Z/(mn),K_1(A)) \ar[r]^-{\psi_{mn}} & \Z/(mn) }
\end{equation}
for each $m,n$, where the vertical maps are those induced by the morphisms of Examples \ref{lambda ex} part \eqref{kap mn n}, and are just the canonical inclusions (see Lemma \ref{kap lem} and the computations in Examples \ref{lambda ex} part \eqref{kap mn n}).   Let now $J$ be as in line \eqref{j set}, and use the maps $\kappa_{mn,n}$ to make the collections $(\Z/n)_{n\in J}$ and $(\text{Tor}(\Z/n,K_1(A)))_{n\in J}$ into directed systems.  As the diagrams in line \eqref{div} commute, taking direct limits gives a map $\psi:\tor K_1(A)\to \Q/\Z$.  This is our desired element of $\text{Hom}(\tor K_1(A),\Q/\Z)$ as in the statement.  

At this point we have a well-defined map
\begin{equation}\label{desid}
\text{ker}(f_\Lambda)\to \text{Hom}(K_1(A),\Q/\Z),\quad (\phi_{i,n})_{i\in \Z/2,n\in \N}\mapsto \psi,
\end{equation}
which one directly checks is a homomorphism.  Lemma \ref{annoying algebra} implies that it is bijective (we use here that the canonical maps $\text{Tor}(\Z/n,K_1(A))\to \tor K_1(A)$ and $\Z/n\to \Q/\Z$ are injective for all $n$).   The map in line \eqref{desid} is moreover a homeomorphism as the topologies are those of pointwise convergence.

Finally, we check commutativity of the left hand square in the diagram in the statement.  Note first that that the mapping cone $C(\iota_Q)$ of the inclusion $\iota_Q:\C\to Q$ is the direct limit over the directed set $J$ (see line \eqref{j set}) of the mapping cones $C(\iota_n)$ of the inclusions $\iota_n:\C\to M_n(\C)$.  Let then $\alpha\in k^0(A)$.  Then the map 
$$
K_0(A;\Q/\Z)\to K_0(\C;\Q/\Z)=\Q/\Z
$$
induced by $\alpha$ is the direct limit over $J$ of the maps 
$$
K_0(A;\Z/n)\to K_0(\C;\Z/n)=\Z/n
$$
induced by $\alpha$.  Comparing the image of $\underline{\gamma}(\alpha)$ under the map in line \eqref{desid} to the map $\delta(\alpha)$ from Definition \ref{sec pair def}, we see that this implies they are the same.
\end{proof}

Under the presence of the UCT, the secondary pairing has particularly good properties.  

\begin{corollary}\label{uct pair}
Let $A$ be a separable $C^*$-algebra that satisfies the UCT.  Then the secondary pairing map
$$
\delta:kl^i(A)\to \text{Hom}(\tor K_{i+1}(A),\Q/\Z)
$$
is a topological isomorphism for $i\in \{0,1\}$.  

Moreover, for $i\in \{0,1\}$ the map $\delta$ identifies $kl^i(A)$ isomorphically and homeomorphically with the Pontrjagin dual of $\tor K_{i+1}(A)$.   In particular, if $K^*(A)$ is finitely generated, then the torsion subgroup of $K^i(A)$ identifies naturally with the Pontrjagin dual of the torsion subgroup of $K_{i+1}(A)$.
\end{corollary}

Before proving this, let us note that related results have been known for a long time.  For example, in \cite[pages 62-3]{Brown:1984rx}, Brown essentially shows that for a nuclear, quasidiagonal, UCT $C^*$-algebra, one has a natural isomorphism 
$$
kl^1(A)\cong \text{Ext}(\tor K_0(A),\Z). 
$$
Thanks to the isomorphism of line \eqref{ext def 3} above, this is closely related to the result of Corollary \ref{uct pair}.

\begin{proof}[Proof of Corollary \ref{uct pair}]
Looking at the diagram in the statement of Theorem \ref{totk}, the two left vertical maps factor through $KL$-theory, so we get a commutative diagram
$$
\xymatrix{ 0 \ar[r] & kl^0(A) \ar[r] \ar[d]^-\delta & KL^0(A) \ar[d]^-{\underline{\gamma}} \ar[r]^-\gamma & \text{Hom}(K_0(A),\Z) \ar@{=}[d] & \\
0\ar[r] & \text{Hom}(\tor K_1(A),\Q/\Z)\ar[r] & \text{Hom}_{\Lambda}(\underline{K}(A),\underline{K}(\C)) \ar[r]^-{f_\Lambda} & \text{Hom}(K_0(A),\Z) \ar[r] &  0 }.
$$
The UCT implies that $\underline{\gamma}$ is a topological isomorphism by \cite[Theorem 4.1]{Dadarlat:2005aa}, and (therefore, or directly) that the map $\gamma$ is surjective.  Hence $\delta:kl^0(A)\to \text{Hom}(\tor K_1(A),\Q/\Z)$ is a topological isomorphism by the five lemma.  The statement in the other degree follows on replacing $A$ with its suspension.

To see the statement about Pontrjagin duals, recall that the Pontrjagin dual of the discrete group $\tor K_i(A)$ is the group $\text{Hom}(\tor K_i(A),\R/\Z)$, equipped with the topology of pointwise convergence when $\R/\Z$ has its usual topology.  However, as $\tor K_i(A)$ is a torsion group, any homomorphism $K_i(A)\to \R/\Z$ takes values in the torsion subgroup $\Q/\Z$ of $\R/\Z$, so the canonical inclusion 
\begin{equation}\label{qz to rz}
\text{Hom}(\tor K_i(A),\Q/\Z)\to \text{Hom}(\tor K_i(A),\R/\Z)
\end{equation}
is an isomorphism.  As the topology of pointwise convergence on $\text{Hom}(\tor K_i(A),\Q/\Z)$ is defined using the discrete topology on $\Q/\Z$, and not the topology it inherits from $\R/\Z$, it is not immediate that the map in line \eqref{qz to rz} is a homeomorphism.  However, this is indeed the case: the point is that any finite subset $F$ of $\tor K_i(A)$ is contained in a finite subgroup, say $H$.  Let $N=|H|$ be the cardinality of $H$, whence any homomorphism $\tor K_i(A) \to \Q/\Z$ must take $F$ into the finite subgroup $\Z/N\cong (\Z\cdot (1/N))/\Z$ of $\Q/\Z$; as $(\Z\cdot (1/N))/\Z$ inherits the discrete topology from $\R/\Z$, the inclusion in line \eqref{qz to rz} is indeed a homeomorphism.

The second result follows as finite generation of $K^*(A)$ (or just countability: see for example \cite[Corollary 7.9]{Willett:2020aa}) implies that $K^*(A)=KL^*(A)$, even without the UCT.  Moreover if we assume the UCT exact sequence 
$$
0\to \text{Ext}(K_{i+1}(A),\Z)\to K^i(A)\to \text{Hom}(K_i(A),\Z)\to 0
$$
holds for $i\in \{0,1\}$, then finite generation of $K^*(A)$ (plus a little algebra) implies that $kl^*(A)$ is exactly the torsion subgroup of $KL^*(A)$.
\end{proof}

We give one more corollary about the structure of $K$-homology, purely for interest.  This was already essentially known a long time ago (see \cite[Section 5, page 977]{Brown:1973aa}\footnote{More precisely, this reference covers the commutative case, which is equivalent to knowing the result in the UCT case by continuity of the Kasparov product.}), but with a different approach.

\begin{corollary}[Brown-Douglas-Fillmore]\label{max cpt}
Let $A$ be a separable $C^*$-algebra satisfying the UCT.  Then for $i\in \{0,1\}$, $kl^i(A)$ is the unique maximal compact subgroup of $KL^i(A)$, and $k^i(A)$ is the unique maximal compact subgroup of $K^i(A)$.
\end{corollary}

\begin{proof}
Corollary \ref{uct pair} identifies $kl^i(A)$ with the Pontrjagin dual of a discrete group, whence it is compact.  To show that it is the unique maximal compact subgroup, it suffices to show that it contains any other compact subgroup of $KL^i(A)$.  This follows as any compact subgroup must be contained in the kernel of the map $\gamma:KL^i(A)\to \text{Hom}(K_i(A),\Z)$.  The statement for $K^i(A)$ follows as the kernel of the canonical quotient maps $K^i(A)\to KL^i(A)$ and $k^i(A)\to kl^i(A)$ are both the closure of zero, which has the indiscrete topology.
\end{proof}

It would be interesting to know if this holds without the UCT.  It remains true in general that $k^i(A)$ contains any compact subgroup of $K^i(A)$, but we do not know if $k^i(A)$ itself is compact.

\section{Relative eta invariants}\label{eta sec}

In this section we relate the secondary pairing of Definition \ref{sec pair def} to relative eta invariants in the sense of Atiyah-Patodi-Singer \cite{Atiyah:1975aa}.  The main ideas come from work of Antonini-Azzali-Skandalis \cite{Antonini:2014aa}: our results here mainly consist largely of fitting their ideas into our context, and we do not claim much originality.  



First we recall some background on flat bundles: we use \cite{Steenrod:1951aa} as a general reference; see also for example \cite[Section 3.11.1]{Gilkey:1994aa} for some discussion in the manifold case in terms in terms of vanishing curvature.  Let $X$ be a Hausdorff topological space.  Recall that a complex vector bundle\footnote{This description more properly defines a \emph{Hermitian} vector bundle because we are using $U_n(\C)$ rather than $GL_n(\C)$, but the two notions are equivalent up to isomorphism: see for example \cite[Corollary 12.8]{Steenrod:1951aa}.} of rank $n$ over $X$ can be described by the following data: 
\begin{itemize}
\item an open cover $\{U_i\}_{i\in I}$ of $X$; 
\item for each $(i,j)\in I\times I$ such that $U_i\cap U_j\neq \varnothing$, a continuous function $g_{ij}:U_i\cap U_j\to U_n(\C)$ such that whenever $U_i\cap U_j\cap U_k\neq \varnothing$, we have the \emph{cocycle identity}
\begin{equation}\label{cocycle}
g_{ij}g_{jk}=g_{ik}
\end{equation}
of functions $U_i\cap U_j\cap U_k\to U_n(\C)$.  
\end{itemize}
A \emph{flat (complex) vector bundle} of rank $n$ over $X$ is described in exactly the same way, with the additional assumption that the $g_{ij}$ are locally constant; equivalently, the $g_{ij}$ are continuous when we replace $U_n(\C)$ with $U_n(\C)_d$, which is the same underlying group equipped with the discrete topology.  There is a natural notion of equivalence for such bundles with structure group $U_n(\C)_d$, and we will consider flat bundles as equivalence classes of data as above: see \cite[2.8]{Steenrod:1951aa} and the discussion starting in the last paragraph of \cite[page 15]{Steenrod:1951aa}.

Assume now that $X$ is also path connected, locally path connected, and semi-locally simply connected: for example this holds if $X$ is a connected CW complex, or connected manifold, which are the only examples we will use.  Then \cite[13.9]{Steenrod:1951aa} gives a one-to-one correspondence between equivalence class of flat complex vector bundles over $X$ of rank $n$ and equivalence classes of homomorphisms $\pi_1(X)\to U_n(\C)$ from the fundamental group of $X$ to $U_n(\C)$ modulo unitary equivalence.  This correspondence can be described concretely as follows.
\begin{itemize}
\item Given a homomorphism $\phi:\pi_1(X)\to U_n(\C)$, let $\widetilde{X}$ be the universal cover of $X$, and define the \emph{associated flat bundle}
\begin{equation}\label{ass flat}
V_\phi:=\widetilde{X}\times_\phi \C^n,
\end{equation}
where the right hand side means $\widetilde{X}\times \C^n$ modulo the equivalence relation $(gx,v)\sim (x,\phi(g)^{-1}v)$ for all $x\in \widetilde{X}$ and $v\in \C^n$.
\item Given an open cover $\{U_i\}_{i\in I}$ of $X$ and continuous functions $g_{ij}:U_i\cap U_j\to U_n(\C)_d$ satisfying the identity in line \eqref{cocycle}, let $\gamma:[0,1]\to X$ be a path representing an element of $\pi_1(X)$.  Choose elements $0=t_0<t_1<\cdots <t_N=1$ of $[0,1]$ such that each restriction $\gamma|_{[t_l,t_{l+1}]}$ has image entirely contained in some $U_{i(l)}$, and define 
$$
\phi:\pi_1(X)\to U_n(\C),\quad [\gamma]\mapsto g_{i(N)i(N-1)}(\gamma(t_{N}))\cdots g_{i(1)i(0)}(\gamma(t_1))
$$
(all the choices involved here only affect the outcome by conjugation by an element of $U_n(\C)$: compare \cite[13.4-13.6]{Steenrod:1951aa}).
\end{itemize} 

We will also briefly need flat bundles over $X$ whose fibre is a (topological) module $M$ over some (topological) ring other than $\C$.  The definition is essentially the same: let $G$ be the group of continuous module automorphisms of $M$ equipped with the discrete topology; a flat bundle then corresponds to a choice of an open cover $\{U_i\}_{i\in I}$ of $X$ and continuous (i.e.\ locally constant) maps $g_{ij}:U_i\cap U_j\to G$ satisfying the cocycle identity from line \eqref{cocycle}.  If $X$ is path connected, locally path connected, and semi-locally simply connected, then equivalence classes of flat bundles are again in one-to-one correspondence with homomorphisms $\pi_1(X)\to G$ up to inner automorphism.

We now recall some background on Dirac-type operators and eta invariants.  We refer to \cite[Section 11.1]{Higson:2000bs}, \cite[Section 3.3]{Berline:2004aa}, and \cite[Chapter 3]{Roe:1998ad} for general background on Dirac-type operators; we summarize this here for completeness.  Let $M$ be a closed Riemannian manifold.  Following \cite[Definition 11.1.2]{Higson:2000bs}, a \emph{Dirac bundle}\footnote{Also called a \emph{Clifford bundle} in many references, such as \cite[Definition 6.1]{Antonini:2014aa}.  It is also often assumed that a Dirac bundle is $\Z/2$-graded, but that is for the even parity case; we work in the odd parity case, so give an ungraded version.} over $M$ is a complex Hermitian vector bundle $S$ over $M$ equipped with a(n $\R$-linear) bundle map 
$$
c:T^*M\to \text{End}(S)
$$
such that for each $x\in M$ and $v\in T_x^*M$, the endomorphism $c(v)\in \text{End}(S_x)$ is skew-adjoint, and satisfies $c(v)^2=-\|v\|^21_{S_x}$.  A \emph{Dirac-type operator} $D$ is any first order, symmetric, partial differential operator acting on the space $C^\infty(M;S)$ of sections of a Dirac bundle $S$, and with the property that if $f\in C^\infty(M)$ and $M_f:C^\infty(M;S)\to C^\infty(M;S)$ is the corresponding multiplication operator, then 
$$
[D,M_f]s=c(df)s
$$  
for all $s\in C^\infty(M;S)$.  Given a Dirac bundle, one can always build a Dirac operator on it using an appropriate connection: see for example \cite[Proposition 3.42]{Berline:2004aa}.  A Dirac operator $D$ as we have defined it (in particular, $S$ is ungraded) canonically determines a $K$-homology class : see for example \cite[Chapter 10 and Section 11.1]{Higson:2000bs}.

Let then $D$ be a Dirac-type operator on a closed Riemannian manifold $M$.  We define $\eta(D)$ to be the Atiyah-Patodi-Singer eta invariant of $D$: the original reference for this is \cite{Atiyah:1975ab}, and a textbook exposition can be found for example in \cite[Section 1.13]{Gilkey:1994aa}.  The exact definition will not be important for us, but suffice to say that it is a measure of the spectral asymmetry of $D$, defined using a zeta function-type regularization.

Let now $V$ be a flat bundle and $D$ be a Dirac-type operator acting on a Dirac bundle $S$ over a closed manifold $M$.  We can then form the twisted operator $D_V$ acting on $S\otimes V$, which will still be a Dirac operator: compare for example \cite[Example 3.24]{Roe:1998ad} and \cite[Section 3.11.2]{Gilkey:1994aa}.  As in \cite[Lemma 3.11.1]{Gilkey:1994aa} the eta invariant $\eta(D_V)$ depends on $D$ and the flat structure of $V$; the notation is maybe a bit misleading as different flat structures on the same vector bundle $V$ will in general give different values for $\eta(D_V)$.   Given then two flat bundles $V$ and $W$ on $M$, the \emph{relative eta invariant} is by definition the equivalence class in $\R/\Z$ of the real number
\begin{equation}\label{rel eta}
\rho_{V,W}(D):=\frac{1}{2}\Big(\eta(D_{V})-\eta(D_{W})\Big)-\frac{1}{2}\Big(\text{dim}(\text{ker}(D_{V}))-\text{dim}(\text{ker}(D_{W}))\Big)
\end{equation}
in $\R/\Z$ (one takes the equivalence class mod $\Z$ as this has much better invariance properties).  In general, it depends on the flat structures on $V$ and $W$, not just on the underlying vector bundles.

We now recall some background on the work of Antonini-Azzali-Skandalis.  Let $M$ be a closed manifold, and let $B$ be a II$_1$-factor.  Following \cite[Section 2.3]{Antonini:2014aa}, elements of $K^1(M;\R/\Z)$ can be described as triples $(E,F,w)$, where $E$ and $F$ are finitely generated projective modules over $C(M)$ (i.e.\ modules of sections of vector bundles over $M$), and $w:E\otimes B\to F\otimes B$ is an isomorphism (of finitely generated, projective $B$-modules); the class of such a triple is called a \emph{relative class}.  

Assume now that $M$ is connected, and let $\Gamma=\pi_1(M)$.  Let $\pi:\Gamma\to U_n(\C)$ be a unitary representation with associated flat bundle $V_\pi$ as in line \eqref{ass flat} and let $V_n$ be the trivial (flat) bundle of the same dimension.  Write $E_\pi$ and $1_n$ for the (finitely generated, projective) $C(M)$-modules of sections  of $V_\pi$ and $V_n$ respectively.  Then as in \cite[Proposition 5.2]{Antonini:2014aa}, there is a II$_1$-factor $B$ and a class of isomorphisms $w:E_\pi\otimes B\to 1_n\otimes B$ such that the relative class $[E_\pi,1_n,w]$ in $K^1(M;\R/\Z)$ does not depend on the choices involved when constructing $B$ or $w$.  As in \cite[Definition 5.3]{Antonini:2014aa}, we define 
\begin{equation}\label{pi rz class}
[\pi]:=[E_\pi,1_n,w]\in K^1(M;\R/\Z).  
\end{equation}

Antonini-Azzali-Skandalis \cite[Proposition 6.3]{Antonini:2014aa} then reformulate an important result of Atiyah-Patodi-Singer \cite[Section 5]{Atiyah:1976aa} in this language.  

\begin{theorem}[Antonini-Azzali-Skandalis] \label{aas the}
Let $M$ be a closed, connected Riemannian manifold.  Let $D$ be a Dirac-type operator representing a class $[D]\in K_1(M)$.  Let $V,W$ be flat bundles over $M$ associated to representations $\pi_V,\pi_W:\pi_1(M)\to U_r(\C)$, and let $[\pi_V],[\pi_W]\in K^1(M;\R/\Z)$ be the associated classes as in line \eqref{pi rz class}.  Let $\rho_{V,W}(D)$ be as in line \eqref{rel eta}.  Then 
$$
\langle [D],[\pi_V]-[\pi_W]\rangle=\rho_{V,W}(D).\eqno\qed
$$
\end{theorem}

We are now finally ready to relate the secondary pairing and the eta invariant.

\begin{corollary}\label{aas cor}
Let $M$ be a closed manifold, let $V$ and $W$ be flat vector bundles over $M$ of the same rank, let $D$ be a Dirac operator on some bundle over $M$ representing a torsion class in $K_1(M)$, and let $\rho_{V,W}(D)$ be the relative eta invariant of line \eqref{rel eta}.    Then
$$
\langle [D],[V]-[W]\rangle_{\text{II}}=\rho_{V,W}(D).
$$ 
\end{corollary}

\begin{proof}
If $M$ has connected components $M_1,...,M_N$, then we have that 
$$
\langle [D],[V]-[W]\rangle_{\text{II}}=\sum_{i=1}^N \langle [D|_{M_1}],([V]-[W])|_{M_i}\rangle_{\text{II}}.
$$
Working one connected component at a time, we may thus assume that $M$ is connected.  

Let then $\pi_V,\pi_W:\Gamma\to U_r(\C)$ be representations such that $V$ and $W$ arise by the associated bundle construction as in line \eqref{ass flat} above.  Let $[\pi_V],[\pi_W]\in K^1(M;\R/\Z)$ be the classes constructed in line \eqref{pi rz class} above.  Let $B$ be the II$_1$ factor used to define $K$-theory with $\R$ and $\R/\Z$ coefficients, and let $\beta:K^1(M;\R/\Z)\to K^0(M)$ be the map on $K$-theory arising from the second arrow below 
$$
0\to C(M)\otimes SB \to C(Y)\otimes C(\iota_B)\to C(M)\to 0.
$$
Then by definition of relative $K$-theory, $\beta$ takes $[\pi_V]-[\pi_W]$ to $[V]-[W]$.  We thus have (by definition: see Remark \ref{rz rem}) that 
$$
\langle [D],[V]-[W]\rangle_{\text{II}}=\langle [D],[\pi_V]-[\pi_W]\rangle
$$
where the right hand side is the result of the primary pairing $K_1(M)\times K^1(M;\R/\Z)\to \R/\Z$.  Applying Theorem \ref{aas the}, we are done.
\end{proof}

\begin{remark}
It is a corollary of Corollary \ref{aas cor} that if $D$ is a Dirac type operator on $M$ such that the class $[D]\in K_1(M)$, then the eta invariant $\rho_{V,W}(D)$ depends only on the class of $[V]-[W]\in K^0(M)$, and in particular only on the isomorphism classes of $V$ and $W$ as \emph{topological} vector bundles.  This is in sharp contrast to the general case where $\rho_{V,W}(D)$ depends (strongly) on the isomorphism classes of $V$ and $W$ as \emph{flat} vector bundles.  For example, let $M=S^1$ and $D=-i\frac{d}{dx}$ be the usual Dirac operator.  Then for any $\theta \in \R/\Z$, \cite[Example 3.11.4]{Gilkey:1994aa} gives a flat structure on the trivial line bundle $\C$ over $S^1$ giving a flat bundle $V_\theta$; these bundles have the property that $\rho_{V_\theta,V_0}(D)=\theta$ (compare also the computations on \cite[pages 410-411]{Atiyah:1975aa}).  

Thus relative eta invariants of Dirac-type operators associated to torsion classes in $K$-homology are much more `robust' than those associated to non-torsion classes.
\end{remark}

Our next goal is to bootstrap Corollary \ref{aas cor} up to a version that works for a finite CW complex $Y$ in place of the manifold $M$.  The strategy will essentially be to reduce to the manifold case using the Baum-Douglas geometric geometric model for $K$-homology \cite{Baum:1980pt}; this is inspired by work of Higson-Roe \cite{Higson:2008qb}.

Our version of the Baum-Douglas model is that used by Higson-Roe in \cite[Section 3]{Higson:2008qb}, which is in turn based on work of Keswani \cite{Keswani:1999aa}.  To set up notation for this, let $Y$ be a finite CW complex.  Then an \emph{odd geometric $K$-cycle} for $X$ is a triple $(M,S,f)$ where: 
\begin{itemize}
\item $M$ is a closed, odd-dimensional (but possibly disconnected and with components of different dimensions) oriented Riemannian manifold; 
\item $S$ is a Dirac bundle on $M$; 
\item $f:M\to X$ is a continuous map.  
\end{itemize}
One can define an equivalence relation on such cycles (the details are not important to us, but see for example \cite[Section 3]{Higson:2008qb}) so that the resulting group is isomorphic to $K_1(Y)$.  Indeed, using the Dirac bundle $S$, one can build an associated Dirac type operator $D_S$ and therefore a $K$-homology class $[D_S]\in K_1(M)$.  The assignment 
\begin{equation}\label{geo an}
(M,S,f)\mapsto f_*[D_S]\in K_1(Y)
\end{equation}
then induces an isomorphism between the Baum-Douglas and Kasparov models of $K$-homology, as shown in \cite{Baum:2007ek}.

The following lemma is a mild generalization of \cite[Lemma 4.1]{Antonini:2014aa}.

\begin{lemma}\label{fac triv}
Let $Y$ be a finite connected CW complex with universal cover $\widetilde{Y}$.  Let $B$ be a II$_1$-factor, let $u:\pi_1(Y)\to U(B)$ be a homomorphism from the fundamental group of $Y$ to the unitary group of $B$, and let $E_u:=\widetilde{Y}\otimes_\Gamma B$ be the associated flat bundle over $Y$ with fibres $B$.  Then the class $[E_u]\in K_0(C(Y)\otimes B)$ is the same as the class $[1_B]$ of the trivial bundle $Y\times B$.
\end{lemma}

\begin{proof}
In \cite[Lemma 4.1]{Antonini:2014aa}, Antonini-Azzali-Skandalis show that this holds when $Y$ is a closed, connected manifold.  Let now $Y$ be a connected finite CW complex.  As $K_0(C(Y)\otimes B)\cong K^0(Y)\otimes \R$ is finitely generated, it suffices to show that for any class $x$ in $K_0(Y)$, the pairings of $[E_u]$ and $[1_B]$ with $x$ are the same.  We may represent $x$ as a Baum-Douglas cycle $[M,S,f]$, in which case naturality reduces us to the case of pairings for $Y=M$, and we are done by \cite[Lemma 4.1]{Antonini:2014aa} (possibly working one connected component at a time if $M$ is disconnected).
\end{proof}

The following result is closely connected to (and in some sense generalizes) \cite[Theorem 6.10]{Willett:2024aa}.

\begin{theorem}\label{eta main}
Let $Y$ be a finite CW complex, let $V$ and $W$ be flat vector bundles over $Y$ of the same dimension, and let $(M,S,f)$ be an odd geometric cycle for $Y$ representing a torsion class in $K_1(Y)$.  Then
$$
\langle [M,S,f],[V]-[W]\rangle_{\text{II}}=\rho_{f^*V,f^*W}(D_S),
$$ 
where $D_S$ is any choice of Dirac-type operator associated to the Dirac bundle $S$.
\end{theorem}

\begin{proof}
As in the proof of Corollary \ref{aas cor}, we may assume that $Y$ is connected.   Let $\Gamma=\pi_1(Y)$ and let $\pi_V,\pi_W:\Gamma\to U_r(\C)$ be representations underlying the flat bundles $V,W$ as in line \eqref{ass flat}.  The construction of Antonini-Azzali-Skandalis from line \eqref{pi rz class} makes sense in this context: the only difference is we have to use Lemma \ref{fac triv} above in place of \cite[Lemma 4.1]{Antonini:2014aa}.  We thus can build associated classes $[\pi_V],[\pi_W]\in K^1(Y;\R/\Z)$ such that the natural map $K^1(Y;\R/\Z)\to K^0(Y)$ takes the difference $[\pi_V]-[\pi_W]$ to $[V]-[W]$.  Just as in the proof of Corollary \ref{aas cor}, we then have 
\begin{align*}
\langle [M,S,f],[V]-[W]\rangle_{\text{II}} & =\langle [M,S,f],[\pi_V]-[\pi_W]\rangle \\
& = \langle [M,S,\text{id}],[\pi_V\circ f_*]-[\pi_W\circ f_*]\rangle
\end{align*}
where the first equality is by definition of the secondary pairing (compare Remark \ref{rz rem}) and the second is naturality of the primary pairing where we have used that the pulback of the class $[\pi_V]$ under $f$ is $[\pi_V\circ f_*]$ with $f_*:\pi_1(M)\to \Gamma$ the map on fundamental groups induced by $f$.  According to the definition of the isomorphism between Baum-Douglas $K$-homology and Kasparov $K$-homology (see line \eqref{geo an} above), $[M,S,f]$ corresponds to $[D_S]$.  Hence 
$$
\langle [M,S,f],[V]-[W]\rangle_{\text{II}}=\langle [D_S],[\pi_V\circ f_*]-[\pi_W\circ f_*]\rangle
$$
and we are done by Corollary \ref{aas cor}.
\end{proof}

\begin{example}\label{non triv eta}
Let us give an example where the pairing of Theorem \ref{eta main} is non-trivial.  

Let $Y$ be a finite, connected CW complex.  Write $\Gamma=\pi_1(Y)$, and assume that the torsion subgroup $\tor \Gamma_{ab}$ of the abelianization of $\Gamma$ is non-trivial: for example, $Y$ might be a non-orientable surface, in which case $\tor\Gamma_{ab}\cong \Z/2$.  

Let $g\in \tor \Gamma_{ab}$ have order $d>1$.  Let $\pi:\langle g\rangle\to U_1(\C)$ be the representation of the cyclic group generated by $g$ sending $g$ to the $d^{\text{th}}$ root of unity $e^{2\pi i /d}$.  Abusing notation, extend $\pi$ to a representation $\pi:\Gamma_{ab}\to U_1(\C)$ using injectivity of $U_1(\C)$; continuing to abuse notation, pull $\pi$ back to a representation $\pi:\Gamma\to U_1(\C)$.  On the other hand, let $\sigma:\Gamma\to U_1(\C)$ be the trivial representation.   

We recall that there is a canonical (injective) map $H_1(Y)\to K_1(Y)$ concretely defined by sending the homotopy class of a loop $f:S^1\to Y$ to the Baum-Douglas cycle $[S^1,\C,f]$: see for example \cite[Definition 3.1, Proposition 3.4, and Theorem 4.1]{Matthey:2002aa}.   Let now $f:S^1\to Y$ be a continuous function realizing $[g]\in H_1(Y)=\Gamma_{ab}$ and let $[S^1,\C,f]$ be the corresponding Baum-Douglas cycle.  

Let now $V_\pi$ and $V_\sigma$ be the line bundles on $Y$ associated to $\pi$ and $\sigma$ respectively as in line \eqref{ass flat}.  Then
$$
\langle [S^1,\C,f],[V_\pi]-[V_\sigma]\rangle_{\text{II}}=\rho_{f^*V_\pi,f^*V_\sigma}(D),
$$
where $D=-i\frac{d}{dx}$ is the usual Dirac operator on the circle.  Following the computations in \cite[Example 3.11.4]{Gilkey:1994aa} or \cite[pages 410-411]{Atiyah:1975aa} we have 
$$
\rho_{f^*V_\pi,f^*V_\sigma}(D)=\Big[\frac{1}{2\pi i} \log \pi(g) \Big] = [1/d]\in \Q/\Z,
$$
getting a non-trivial pairing as claimed.
\end{example}

\section{The Thomsen exact sequence and the zeta map}\label{tz sec}

In this section, we discuss a connection between our secondary pairing and the Thomsen exact sequence of \cite{Thomsen:1995aa}, as well as the \emph{$\zeta$ map} independently due to Carri\'{o}n-Gabe-Schafhauser-Tikuisis-White \cite[Section 3.1]{Carrion:2020aa}, and Gong-Lin-Niu \cite[Remark 6.6]{Gong:2023aa}.  See also \cite{Munkholm:2025aa} for a characterization of these maps that shows they are unique in a fairly strong sense; in particular, this implies that the constructions of Carri\'{o}n et al and of Gong et al give the same map.   Our main motivation is potential applications to group $C^*$-algebras, but we focus here on general theory, and just give one non-trivial example.

We need to recall some background.  Let $A$ be a unital $C^*$-algebra.  For each $n$, let $U_n(A)$ denote the unitary group of $M_n(A)$, and let $\overline{D}_n(A)$ denote the closure of the derived subgroup of $U_n(A)$.  The \emph{Hausdorffized unitary algebraic $K_1$-group} of $A$, denoted $\overline{K}_1^{alg}(A)$, is defined to be direct limit of the groups $U_n(A)/\overline{D}_n(A)$ as $n\to\infty$ through the maps induced by the inclusions
$$
U_n(A)\to U_{n+1}(A),\quad u\mapsto \begin{pmatrix} u & 0 \\ 0 & 1 \end{pmatrix}
$$
(as are typically used to define $K_1(A)$).  

We write $TA$ for the tracial state space of $A$ equipped with the weak-$*$ topology it inherits from $A^*$, and we write $\text{Aff}(TA)$ for the real vector space of continuous affine functions $f:TA\to \R$, equipped with the topology induced by the supremum norm.  We let $\rho:K_0(A)\to \text{Aff}(TA)$ denote the usual pairing map, and let $\overline{\rho(K_0(A))}$ denote the (norm) closure of its image.  There is then a natural exact sequence 
\begin{equation}\label{thom seq}
\xymatrix{0 \ar[r] & \frac{\text{Aff}(TA)}{\overline{\rho(K_0(A))}} \ar[r] &  \overline{K_1}^{alg}(A) \ar[r] &  K_1(A) \ar[r] & 0 }
\end{equation}
due to Thomsen  \cite{Thomsen:1995aa} (see also \cite[Section 2.2]{Carrion:2020aa} for an exposition of the sequence in the form in line \eqref{thom seq}), where the map $\overline{K_1}^{alg}(A) \to K_1(A)$ is the canonical forgetful map sending the class of a unitary $u$ to the class of the same unitary.

We recall now that for each $n\geq 2$ there is a homomorphism 
$$
\zeta_n:K_0(A;\Z/n)\to \overline{K_1}^{alg} (A)
$$
defined in \cite[Section 3.1]{Carrion:2020aa} and \cite[Remark 6.6]{Gong:2023aa}. As we will need it later, we briefly recall the definition from \cite{Carrion:2020aa}.  By definition, we have that $K_0(A;\Z/n):=K_1(A\otimes I_n)$, where $I_n$ is the dimension drop algebra of line \eqref{dd alg}, so
$$
A\otimes I_n=\{f\in C([0,1],A\otimes M_n(\C))\mid f(0)=0,~f(1)\in A\otimes 1_{M_n(\C)}\}.
$$
We may represent an element of $K_1(A\otimes I_n)$ via a smooth unitary-valued map $u:[0,1]\to M_{kn}(A)$ for some $k$ such that $u(0)\in M_k(\C1_A)$ and with $u(1)\in M_k(A)$.  Carri\'{o}n et al then define
\begin{equation}\label{zn def}
\zeta_n[u]:=[u(0)]-[u(1)]-\Big[\exp\Big(\frac{2\pi i}{n}\Delta(u)\Big)\Big]
\end{equation}
where $u(0)$ and $u(1)$ are considered as elements of $M_k(A)$, and 
$$
\Delta(u):=\frac{1}{2\pi i} \int_0^1 u'(t)u^*(t)dt
$$
is the de la Harpe-Skandalis determinant \cite{Harpe:1984aa} of the smooth unitary path, an element of the self-adjoint part of $M_{kn}(A)$.  The map $\zeta_n$ is shown to be well-defined and natural in \cite[Section 3.1]{Carrion:2020aa}.  It is shown to be unique in a fairly strong sense in \cite{Munkholm:2025aa}: in particular, this shows that it does not depend on the particular model we chose for $K_0(A;\Z/n)$.

It will be useful for us to combine the maps $\zeta_n$ into a single homomorphism.  For each $n,m\geq 2$, let $\kappa_{mn,n}$ be  as in Examples \ref{lambda ex} part \eqref{kap mn n}.  Then \cite[Proposition 3.2]{Carrion:2020aa} implies that the zeta maps fit into a commutative diagram
$$
\xymatrix{ K_0(A;\Z/n) \ar[r]^-{\zeta_n} \ar[d]^-{\kappa_{mn,n}} & \overline{K_1}^{alg} (A) \ar@{=}[d] \\
K_0(A;\Z/(nm)) \ar[r]^-{\zeta_{mn}} & \overline{K_1}^{alg} (A) }.
$$
Hence we may take the direct limit over the directed set $J$ of line \eqref{j set} to get a natural homomorphism 
\begin{equation}\label{qz zeta}
\zeta:K_0(A;\Q/\Z)\to \overline{K}_1^{alg}(A).
\end{equation}
We now get to our first useful technical tool.  For the statement, recall that if $A$ is a $C^*$-algebra and we tensor $A$ by the short exact sequence in line \eqref{qz ses}, then (using also lines \eqref{qz coeffs} and \eqref{q coeffs}) we get a long exact sequence 
$$
\xymatrix{ K_1(A)\otimes \Q \ar[r] & K_1(A;\Q/\Z) \ar[r] & K_0(A) \ar[d] \\
K_1(A) \ar[u] & K_1(A;\Q/\Z) \ar[l] & K_0(A)\otimes \Q) \ar[l] },
$$
where the maps $K_i(A)\to K_i(A)\otimes \Q$ are given by $x\mapsto x\otimes 1_\Q$.  Splitting this up along the kernel and cokernel of the maps $K_i(A)\to K_i(A)\otimes \Q$, we get the \emph{K\"{u}nneth}, or \emph{change of coefficients} exact sequence
\begin{equation}\label{kcc}
\xymatrix{ 0 \ar[r] & K_0(A)\otimes (\Q/\Z) \ar[r] & K_0(A;\Q/\Z) \ar[r]  & \tor K_1(A)  \ar[r] & 0 }.
\end{equation}
for $i\in \{0,1\}$.

\begin{lemma}\label{zeta diag}
For any $C^*$-algebra $A$ there is a natural commutative diagram
$$
\xymatrix{ 0 \ar[r] & K_0(A)\otimes (\Q/\Z) \ar[r] \ar[d] & K_0(A;\Q/\Z) \ar[r] \ar[d]^-\zeta & \tor K_1(A) \ar@{=}[d] \ar[r] & 0 \\
0 \ar[r] & \tor\Bigg(\frac{\text{Aff}(TA)}{\overline{\rho(K_0(A))}}\Bigg) \ar[r] & \tor ( \overline{K_1}^{alg}(A)) \ar[r] &  \tor K_1(A) \ar[r] & 0 }
$$
of short exact sequences, where the top line is as in line \eqref{kcc}, the bottom line consists of the torsion subgroups of the Thomsen exact sequence from line \eqref{thom seq}, and the left hand vertical map is induced from the pairing map $\rho:K_0(A)\to \text{Aff}(TA)$ (plus the fact that $\text{Aff}(TA)$ is an $\R$-vector space).
\end{lemma}

\begin{proof}
For each $n$ in the set $J$ from line \eqref{j set}, \cite[Proposition 3.2]{Carrion:2020aa} gives a commutative diagram
$$
\xymatrix{ 0 \ar[r] & K_0(A)\otimes (\Z/n) \ar[r] \ar[d] & K_0(A;\Z/n) \ar[r] \ar[d]^-{\zeta_n} & \text{Tor}(\Z/n,K_1(A)) \ar[d] \ar[r] & 0 \\
0 \ar[r] & \frac{\text{Aff}(TA)}{\overline{\rho(K_0(A))}} \ar[r] & \overline{K_1}^{alg}(A) \ar[r] &  K_1(A) \ar[r] & 0 }
$$
where: the top line is the analog of line \eqref{kcc} above, but starting with the short exact sequence in line \eqref{dd ses} rather than that of line \eqref{qz ses}; left hand vertical map is determined by $[p]\otimes [m]\mapsto \frac{m}{n}\rho[p]$; and the right hand map is the canonical inclusion of the $n$-torsion subgroups (compare line \eqref{tor zn}).  Taking the direct limit of the top line over the directed set $J$ gives a commutative diagram of short exact sequences
$$
\xymatrix{ 0 \ar[r] & K_0(A)\otimes (\Q/\Z) \ar[r] \ar[d] & K_0(A;\Q/\Z) \ar[r] \ar[d]^-\zeta & \tor K_1(A) \ar[d] \ar[r] & 0 \\
0 \ar[r] & \frac{\text{Aff}(TA)}{\overline{\rho(K_0(A))}} \ar[r] & \overline{K_1}^{alg}(A) \ar[r] &  K_1(A) \ar[r] & 0 }
$$
where the top line is as in line \eqref{kcc} (we use here that $C(\iota_Q)=\lim_J C(\iota_n)$ in a way compatible with the short exact sequences of lines \eqref{qz ses} and \eqref{dd ses}).  Note that the bottom line splits (non-canonically) as the left hand group is divisible, whence injective.  Hence the bottom row remains exact if we take torsion subgroups.  Moreover, the top row consists of torsion groups, so the vertical maps take image in the torsion subgroups of the bottom row.
\end{proof}

\begin{lemma}\label{base qz}
For any matrix algebra $M_n=M_n(\C)$, let $\phi:K_0(M_n;\Q/\Z)\to \Q/\Z$ be the canonical isomorphism arising from the fixed isomorphism $K_1(C(\iota_Q))\cong \Q/\Z$ (compare line \eqref{qz coeffs}), and define $d:\tor \overline{K}_1^{alg}(M_n)\to \Q/\Z$ by 
$$
d([u])=\frac{1}{2\pi i}\log\det(u).
$$
Then $d$ is well-defined, the diagram
$$
\xymatrix{ K_0(M;\Q/\Z) \ar[d]^-\zeta \ar[r]^-\phi & \Q/\Z \ar@{=}[d] \\
 \tor \overline{K}_1^{alg}(M_n) \ar[r]^-d & \Q/\Z }
$$
commutes, and all the maps appearing are isomorphisms.
\end{lemma}

\begin{proof}
As the kernel of the determinant map $\det:U_n(\C)\to U_1(\C)$ is exactly the commutator subgroup for each $n>1$, $d$ is a well-defined isomorphism.  It thus suffices to show that the diagram commutes.  For this, it suffices to check that the corresponding diagram
$$
\xymatrix{ K_0(M_n;\Z/m) \ar[d]^-{\zeta_n} \ar[r]^-{\phi_m} & \Z/m \ar[d] \\
 \tor \overline{K}_1^{alg}(M_n) \ar[r]^-d & \Q/\Z }
$$
commutes for any $m$, where the right hand vertical map takes $[1]$ to $[1/m]$ and $\phi_m$ is again the canonical isomorphism.  As $K_0(M_n;\Z/m)\cong \Z/m$, it suffices to check this on a generator.

Now, a generator of $K_0(M_n;\Z/m)$ can be given by a function $u:[0,1]\to M_{mn}(\C)$ of the form $u=u_{11}\oplus 1_{mn-1}$, where $u_{11}:[0,1]\to S^1$ is a smooth function with $u_{11}(0)=1=u_{11}(1)$, and with winding number one\footnote{This follows from considering the short exact sequence $0\to M_n(SM_m)\to M_n(I_m)\to M_n\to 0$ defined by taking matrices over line \eqref{dd ses}, and noting that one can choose a generator of $M_n(SM_m)=M_{nm}(C_0(0,1))$ of the stated form.}.  According to line \eqref{zn def}, we have then that 
$$
\zeta_n[u]=\exp\Bigg(\frac{1}{m}\int_0^1 u'(t)u^*(t)dt\Bigg)=e^{2\pi i/m}.  
$$
Hence $d(\zeta_n[u])=[1/m]$, and we are done.
\end{proof}

Here is the main technical result of this section; it generalizes Theorem \ref{intro zeta} from the introduction.  For the statement, recall that a homomorphism $\pi:A\to M_n(\C)$ from a $C^*$-algebra to a matrix algebra canonically induces an element of the $K$-homology group $K^0(A)$ in such a way that the pairing map $\gamma:K^0(A)\to \text{Hom}(K_0(A),\Z)$ takes $[\pi]$ to the map $\pi_*$ it induces on $K$-theory.  Recall also that $k^0(A)$ denotes the kernel of $\gamma$, and the notation for the secondary pairing from Definition \ref{sec pair def}.

\begin{theorem}\label{thom comp}
Let $A$ be a unital $C^*$-algebra.  Let $\pi,\sigma:A\to M_n(\C)$ be homomorphisms that induce the same map on $K$-theory, and let $[\pi]-[\sigma]$ denote the corresponding class in $k^0(A)$.   Then for any unitary $u\in M_n(A)$ such that the class $[u]_{alg}\in \overline{K}_1^{alg}(A)$ is torsion\footnote{Whence also the class $[u]\in K_1(A)$ is torsion, so the secondary pairing below makes sense.}, we have 
$$
\langle [\pi]-[\sigma],[u]\rangle_{\text{II}}=\frac{1}{2\pi i} \log\det(\pi(u)\sigma(u^*)).
$$
\end{theorem}

\begin{proof}
For simplicity, write $M_n$ for $M_n(\C)$, and write 
$$
[\pi]_{alg}:\overline{K}_1^{alg}(A)\to \overline{K}_1^{alg}(M_n)
$$
 for the homomorphism induced by $\pi$, and similarly for $\sigma$.  Let $x\in K_0(A;\Q/\Z)$ be any lift of $[u]\in \tor K_1(A)$ to $K_0(A;\Q/\Z)$ under the quotient map in line \eqref{kcc}.  Then by definition of the secondary pairing we have 
\begin{equation}\label{tc1}
\langle [\pi]-[\sigma],[u]\rangle_{\text{II}}=\phi([\pi\otimes \text{id}_{C(\iota_Q)}(x)]-[\sigma\otimes \text{id}_{C(\iota_Q)}(x)])
\end{equation}
where $C(\iota_Q)$ is the mapping cone of $\iota_Q$ as in line \eqref{cone qz}, and $\phi:K_0(M_n;\Q/\Z)\to \Q/\Z$ is the canonical isomorphism as in Lemma \ref{base qz}.  On the other hand, we have a diagram 
$$
\xymatrix{ K_0(A;\Q/\Z) \ar[rrrr]^-{[\pi\otimes \text{id}_{C(\iota_Q)}]-[\sigma\otimes \text{id}_{C(\iota_Q)}]} \ar[d]^-\zeta & & & & K_0(M_n;\Q/\Z) \ar[d]^-\zeta \ar[r]^-\phi & \Q/\Z \ar@{=}[d]  \\
\tor \overline{K}_1^{alg}(A) \ar[rrrr]^-{[\pi]_{alg}-[\sigma]_{alg}} & & & & \tor \overline{K}_1^{alg}(M_n) \ar[r]^-d & \Q/\Z },
$$
where the first square commutes by naturality of $\zeta$, and the second square is as in Lemma \ref{base qz} (and in particular, commutative).  Hence we have 
\begin{equation}\label{tc2}
\phi([\pi\otimes \text{id}_{C(\iota_Q)}(x)]-[\sigma\otimes \text{id}_{C(\iota_Q)}(x)]) = d(([\pi]_{alg}-[\sigma]_{alg})(\zeta(x))).
\end{equation}

Now, we claim that as $[\pi],[\sigma]\in K^0(A)$ induce the same map $K_0(A)\to \Z$, then they induce the same map  
$$
\tor\Bigg( \frac{\text{Aff}(TA)}{\overline{\rho(K_0(A))}}\Bigg)  \to \tor  \Bigg(\frac{\text{Aff}(TM_n)}{\overline{\rho(K_0(M_n))}}\Bigg). 
$$
Indeed, this follows as we have a natural identification
$$
\tor\Bigg(\frac{\text{Aff}(TA)}{\overline{\rho(K_0(A))}}\Bigg)= \frac{\text{span}_\Q( \overline{\rho(K_0(A))})}{\overline{\rho(K_0(A))}},
$$
and $\pi$ and $\sigma$ agree on the right hand side by assumption that they agree on $K$-theory.  

Now, by naturality of the Thomsen sequence from line \eqref{thom seq}, we have a commutative diagram
$$
\xymatrix{ 0 \ar[r] & \tor\Bigg(\frac{\text{Aff}(TA)}{\overline{\rho(K_0(A))}}\Bigg) \ar[r] \ar[d] & \tor ( \overline{K_1}^{alg}(A)) \ar[r] \ar[d] & \tor K_1(A) \ar[d] \ar[r] & 0 \\
0 \ar[r] & \tor\Bigg(\frac{\text{Aff}(TM_n)}{\overline{\rho(K_0(M_n))}}\Bigg) \ar[r] & \tor ( \overline{K_1}^{alg}(M_n)) \ar[r] &  \tor K_1(M_n) \ar[r] & 0 }
$$
where all vertical maps are induced by the formal difference $[\pi]-[\sigma]$ in the appropriate sense.  From this and the previous claim, for any element $y\in\tor \overline{K}_1^{alg}(A)$, the difference  
$$
([\pi]_{alg}-[\sigma]_{alg})(y)\in  \tor \overline{K}_1^{alg}(M_n)
$$
depends only the image of $y$ in $\tor K_1(A)$.  As the classes $\zeta(x)$ and $[u]_{alg}$ in $\tor \overline{K}_1^{alg}(A)$ both map to $[u]$ in $\tor K_1(A)$, we thus have that 
\begin{equation}\label{tc3}
([\pi]_{alg}-[\sigma]_{alg})(\zeta(x))=([\pi]_{alg}-[\sigma]_{alg})([u]_{alg}).
\end{equation}
The result now follows from lines \eqref{tc1}, \eqref{tc2}, and \eqref{tc3} (and the definition of $d$ from Lemma \ref{base qz}).
\end{proof}

Let us give an example of a non-trivial pairing arising from Proposition \ref{thom comp}.

\begin{example}\label{non triv}
Let $\Gamma$ be a finitely generated torsion-free a-T-menable group such that the torsion subgroup $\tor\Gamma_{ab}$ of the abelianization $\Gamma_{ab}$ is non-trivial.  For example, $\Gamma$ might be the fundamental group of a closed non-orientable  surface, in which case $\tor \Gamma_{ab}\cong \Z/2$.

Let $g\in \tor \Gamma_{ab}$ have order $d>1$.  Let $\pi:\langle g\rangle\to U_1(\C)$ be the representation of the cyclic group generated by $g$ sending $g$ to the $d^{\text{th}}$ root of unity $e^{2\pi i /d}$.  Abusing notation, extend $\pi$ to a representation $\pi:\Gamma_{ab}\to U_1(\C)$ using injectivity of $U_1(\C)$; continuing to abuse notation, pull $\pi$ back to a representation $\pi:\Gamma\to U_1(\C)$.  On the other hand, let $\sigma:\Gamma\to U_1(\C)$ be the trivial representation.  

Now, as $\Gamma$ is a-T-menable, the strong Baum-Connes conjecture holds, and so the maximal assembly map 
$$
\mu:RK_*(B\Gamma)\to K_*(C^*\Gamma)
$$
is surjective (we use here that $\Gamma$ is torsion-free to identify the left hand side of the conjecture with $RK_*(B\Gamma)$).  This implies that $\pi$ and $\sigma$ induce the same map $K_0(C^*\Gamma)\to \Z$ (see for example \cite[Lemma 6.6]{Willett:2024aa}).  Hence Proposition \ref{thom comp} applies to $A=C^*\Gamma$ and the representations $\pi$ and $\sigma$.  

Let $[g]\in K_1(C^*(\Gamma))$ be the class of the canonical unitary associated to any lift of $g$ to $\Gamma$; all such lifts define the same class in both $K_1(C^*\Gamma)$ and $\overline{K}_1^{alg}(C^*\Gamma)$.  Because $g$ has order $d$ in the quotient of $\Gamma$ by its commutator subgroup, $[g]_{alg}$ certainly has order at most $d$ in $\overline{K}_1^{alg}(C^*\Gamma)$, as this is also defined by taking a quotient by (a much larger) group of commutators.  From Proposition \ref{thom comp}, we conclude that 
$$
\langle [\pi]-[\sigma],[g]\rangle_{\text{II}} = [1/d]\in \Q/\Z,
$$
so we get something non-trivial even from this fairly simple one-dimensional situation.  

Note that this computation shows that the class $[g]$ has order exactly $[d]$ in $K_1(C^*\Gamma)$; expanding on the argument slightly, one can conclude that under these assumptions, the canonical map $\tor \Gamma_{ab}\to K_1(C^*\Gamma)$ is injective.  This is well-known, and stronger versions hold in greater generality (see for example \cite[Theorem 1.2]{Matthey:2002ab}), but even this version does not seem to be completely trivial.
\end{example}

The reader may have noticed that Examples \ref{non triv eta} and \ref{non triv} look quite similar\footnote{To the extent that we even copied and pasted some of the text.  This is of course an expository choice we made because we wanted them to look similar, and has nothing to do with laziness.}.  This is not an accident: the two are related by the Baum-Connes assembly map, as we hope to explore further in future work.


\bibliography{Generalbib}

\begin{thebibliography}{10}

\bibitem{Antonini:2014aa}
P.~Antonini, S.~Azzali, and G.~Skandalis.
\newblock Flat bundles, von {N}eumann algebras, and ${K}$-theory with
  $\mathbb{R}/\mathbb{Z}$-coefficients.
\newblock {\em J. K-theory}, 13:275--303, 2014.

\bibitem{Atiyah:1975ab}
M.~Atiyah, V.~K. Patodi, and I.~Singer.
\newblock Spectral asymmetry and {R}iemannian geometry. {I}.
\newblock {\em Math. Proc. Cambridge Philos. Soc.}, 77:43--69, 1975.

\bibitem{Atiyah:1975aa}
M.~Atiyah, V.~K. Patodi, and I.~Singer.
\newblock Spectral asymmetry and {R}iemannian geometry. {II}.
\newblock {\em Math. Proc. Cambridge Philos. Soc.}, 78:405--432, 1975.

\bibitem{Atiyah:1976aa}
M.~Atiyah, V.~K. Patodi, and I.~Singer.
\newblock Spectral asymmetry and {R}iemannian geometry. {III}.
\newblock {\em Math. Proc. Cambridge Philos. Soc.}, 79:71--99, 1976.

\bibitem{Baum:1980pt}
P.~Baum and R.~G. Douglas.
\newblock ${K}$-homology and index theory.
\newblock In {\em Operator algebras and applications, Part I}, volume~38 of
  {\em Proc. Sympos. Pure Math.}, pages 117--173. American Mathematical
  Society, 1980.

\bibitem{Baum:2007ek}
P.~Baum, N.~Higson, and T.~Schick.
\newblock On the equivalence of geometric and analytic ${K}$-homology.
\newblock {\em Pure Appl. Math. Q.}, 3(1):1--24, 2007.

\bibitem{Berline:2004aa}
N.~Berline, E.~Getzler, and M.~Vergne.
\newblock {\em Heat kernels and {D}irac operators}.
\newblock Springer, 2004.

\bibitem{Blackadar:1998yq}
B.~Blackadar.
\newblock {\em K-Theory for Operator Algebras}.
\newblock Cambridge University Press, second edition, 1998.

\bibitem{Brown:1984rx}
L.~G. Brown.
\newblock The universal coefficient theorem for {E}xt and quasidiagonality.
\newblock In {\em Operator algebras and group representations}, volume~I of
  {\em Monogr. Stud. Math. 17}, pages 60--64. Pitman, 1984.

\bibitem{Brown:1973aa}
L.~G. Brown, R.~G. Douglas, and P.~Fillmore.
\newblock Extensions of ${C^*}$-algebras, operators with compact
  self-commutators, and ${K}$-homology.
\newblock {\em Bull. Amer. Math. Soc.}, 79(5):973--978, 1973.

\bibitem{Brown:1977qa}
L.~G. Brown, R.~G. Douglas, and P.~Fillmore.
\newblock Extensions of ${C^*}$-algebras and ${K}$-homology.
\newblock {\em Ann. of Math.}, 105:265--324, 1977.

\bibitem{Carrion:2020aa}
J.~Carri\'{o}n, J.~Gabe, C.~Schafhauser, A.~Tikuisis, and S.~White.
\newblock Classification of $*$-homomorphisms {I}: the simple nuclear case.
\newblock arXiv:2307.06480, 2020.

\bibitem{Dadarlat:2005aa}
M.~Dadarlat.
\newblock On the topology of the {K}asparov groups and its applications.
\newblock {\em J. Funct. Anal.}, 228(2):394--418, 2005.

\bibitem{Dadarlat:1996aa}
M.~Dadarlat and T.~Loring.
\newblock A universal multicoefficient theorem for the {K}asparov groups.
\newblock {\em Duke Math. J.}, 84(2):355--377, 1996.

\bibitem{Dadarlat:2012aa}
M.~Dadarlat and R.~Meyer.
\newblock E-theory for ${C^*}$-algebras over topological spaces.
\newblock {\em J. Funct. Anal.}, 263(1):216--247, 2012.

\bibitem{Harpe:1984aa}
P.~de~la Harpe and G.~Skandalis.
\newblock D\'{e}terminant associ\'{e} \`{a} une trace sur une alg\`{e}bre de
  banach.
\newblock {\em Annales de l'Institut Fourier}, 34(1):241--260, 1984.

\bibitem{Gilkey:1994aa}
P.~Gilkey.
\newblock {\em Invariance Theory, the Heat Equation, and the {A}tiyah-{S}inger
  Index Theorem}.
\newblock CRC Press, second edition, 1994.

\bibitem{Gong:2023aa}
G.~Gong, H.~Lin, and Z.~Niu.
\newblock Homomorphisms into simple $\mathcal{Z}$-stable ${C}^*$-algebras {II}.
\newblock {\em J. Noncommut. Geom.}, 17(3):835--898, 2023.

\bibitem{Higson:2000bs}
N.~Higson and J.~Roe.
\newblock {\em Analytic ${K}$-homology}.
\newblock Oxford University Press, 2000.

\bibitem{Higson:2008qb}
N.~Higson and J.~Roe.
\newblock ${K}$-homology, assembly and rigidity theorems for relative
  eta-invariants.
\newblock {\em Pure Appl. Math. Q.}, 6(2):555--601, 2010.

\bibitem{Kervaire:1963aa}
M.~Kervaire and J.~Milnor.
\newblock Groups of homotopy spheres {I}.
\newblock {\em Ann. of Math.}, 77(3):504--537, 1963.

\bibitem{Keswani:1999aa}
N.~Keswani.
\newblock Geometric ${K}$-homology and controlled paths.
\newblock {\em New York J. Math.}, 5:53--81, 1999.

\bibitem{Matthey:2002aa}
M.~Matthey.
\newblock Mapping the homology of a group to the ${K}$-theory of its
  ${C^*}$-algebra.
\newblock {\em Illinois J. Math}, 46(3):953--977, 2002.

\bibitem{Matthey:2002ab}
M.~Matthey and H.~Oyono-Oyono.
\newblock Algebraic ${K}$-theory in low degree and the {N}ovikov assembly map.
\newblock {\em Proc. London Math. Soc.}, 85(3):43--61, 2002.

\bibitem{Morgan:1974aa}
J.~Morgan and D.~Sullivan.
\newblock The transversality charcteristic class and linking cycles in surgery
  theory.
\newblock {\em Ann. of Math.}, 99(3):463--544, 1974.

\bibitem{Munkholm:2025aa}
M.~Munkholm.
\newblock Uniqueness of the zeta transform in operator ${K}$-theory.
\newblock arXiv:2512.16035v1, 2025.

\bibitem{Roe:1998ad}
J.~Roe.
\newblock {\em Elliptic Operators, topology and asymptotic methods}.
\newblock Chapman and Hall, second edition, 1998.

\bibitem{Rordam:1995aa}
M.~R\o{}rdam.
\newblock Classification of certain infinite simple ${C^*}$-algebras.
\newblock {\em J. Funct. Anal.}, 131:415--458, 1995.

\bibitem{Rordam:2000mz}
M.~R\o{}rdam, F.~Larsen, and N.~Laustsen.
\newblock {\em An Introduction to $K$-Theory for $C^*$-Algebras}.
\newblock Cambridge University Press, 2000.

\bibitem{Rosenberg:1987bh}
J.~Rosenberg and C.~Schochet.
\newblock The {K}\"{u}nneth theorem and the universal coefficient theorem for
  {K}asparov's generalized ${K}$-functor.
\newblock {\em Duke Math. J.}, 55(2):431--474, 1987.

\bibitem{Schochet:1982aa}
C.~Schochet.
\newblock Topological methods for ${C^*}$-algebras {II}: geometric resolutions
  and the {K}\"{u}nneth formula.
\newblock {\em Pacific J. Math.}, 98(2):443--458, 1982.

\bibitem{Steenrod:1951aa}
N.~Steenrod.
\newblock {\em The topology of fibre bundles}.
\newblock Princeton University Press, 1951.

\bibitem{Thomsen:1995aa}
K.~Thomsen.
\newblock Traces, unitary characters, and crossed products by {$\Z$}.
\newblock {\em Publ. RIMS, Kyoto Univ.}, 31:1011--1029, 1995.

\bibitem{Wegge-Olsen:1993kx}
N.~E. Wegge-Olsen.
\newblock {\em K-Theory and $C^*$-Algebras (A Friendly Approach)}.
\newblock Oxford University Press, 1993.

\bibitem{Willett:2024aa}
R.~Willett.
\newblock Conditional representation stability, classification of
  $*$-homomorphisms, and relative eta invariants.
\newblock arXiv:2408.13350, 2024.

\bibitem{Willett:2020aa}
R.~Willett and G.~Yu.
\newblock Controlled ${KK}$-theory and a {M}ilnor exact sequence.
\newblock arXiv:2011.10906. To appear, Doc. Math., 2020.

\end{thebibliography}

\end{document}